\documentclass[11pt]{amsart}
\usepackage{amsmath,amsfonts,amssymb,amsthm,amscd,latexsym,euscript,verbatim}
\usepackage[all]{xy}

\makeatletter
\def\section{\@startsection{section}{1}%
  \z@{1.1\linespacing\@plus\linespacing}{.8\linespacing}%
  {\normalfont\Large\scshape\centering}}
\makeatother

\theoremstyle{plain}

\newtheorem*{thmA}{Theorem A}
\newtheorem*{thmB}{Theorem B}
\newtheorem*{thmC}{Theorem C}

\newtheorem*{thm1.2}{(1.2) Theorem}
\newtheorem*{thm1.3}{(1.3) Theorem}
\newtheorem*{thm1.4}{(1.4) Theorem}
\newtheorem*{prop*}{Proposition}

\newtheorem{prop}{Proposition}[section]
\newtheorem{thm}[prop]{Theorem}
\newtheorem{cor}[prop]{Corollary}
\newtheorem{lemma}[prop]{Lemma}

\theoremstyle{definition}

\newtheorem{Def}[prop]{Definition}
\newtheorem*{Def*}{Definition}

\newtheorem{example}[prop]{Example}

\newtheorem{notation}[prop]{Notation}
\newtheorem*{notation*}{Notation}
\newtheorem{remark}[prop]{Remark}
\newtheorem{remarks}[prop]{Remarks}

\newtheorem*{lemma*}{Lemma}
\newtheorem*{remarks*}{Remarks}

\newtheorem{Defi}{Definition}

\newtheorem{question}[Defi]{Question}

\DeclareMathOperator{\diam}{diam}

\DeclareMathOperator{\Res}{Res}
\DeclareMathOperator{\St}{St}

\newcommand{\cald}{\mathcal{D}}

\newcommand{\calk}{\mathcal{K}}

\newcommand{\calp}{\mathcal{P}}

\newcommand{\ff}{\mathbb{F}}

\newcommand{\pp}{\mathbb{P}}
\newcommand{\qq}{\mathbb{Q}}

\newcommand{\zz}{\mathbb{Z}}

\newcommand{\frakm}{\mathfrak{m}}

\newcommand{\frakP}{\mathfrak{P}}

\newcommand{\ga}{\alpha}
\newcommand{\gb}{\beta}
\newcommand{\gc}{\gamma}
\newcommand{\gC}{\Gamma}

\newcommand{\gd}{\delta}
\newcommand{\gD}{\Delta}
\newcommand{\gre}{\epsilon}

\newcommand{\gvp}{\varphi}

\newcommand{\gs}{\sigma}
\newcommand{\gS}{\Sigma}

\newcommand{\gTH}{\Theta}

\newcommand{\nsg}{\trianglelefteq}

\newcommand{\Dt}{D^{\times}}

\newcommand{\ra}{\rightarrow}

\newcommand{\sminus}{\smallsetminus}
\newcommand{\lan}{\langle}
\newcommand{\ran}{\rangle}

\newcommand{\Aut}{{\rm Aut}}

\newcommand{\In}{{\rm In}_M}
\newcommand{\Inc}{{\rm Inc}_M}

\newcommand{\tensor}{\otimes}

\newcommand{\half}{\textstyle{\frac{1}{2}}}
\newcommand{\restr}{\upharpoonright}

\newcommand{\widebar}[1]{\overset{\mskip1mu\hrulefill\mskip1mu}{#1}
                \vphantom{#1}}

\numberwithin{equation}{section}

\newcommand{\ETH}{$(3\frac{1}{2})$}
\newcommand{\N}{\dot{N}}
\newcommand{\vv}{\vert\vert}

\renewcommand{\restr}{\mathord{\upharpoonright}}
\begin{document}

\title{On graphs and valuations}
\author[I.~Efrat, A.~S.~Rapinchuk, Y.~Segev]{Ido Efrat$^1$\qquad Andrei S.~Rapinchuk$^2$\qquad Yoav Segev}

\address{Ido Efrat \\
         Department of Mathematics \\
         Ben-Gurion University of the Negev\\
         Beer-Sheva 84105 \\
         Israel}
\email{efrat@math.bgu.ac.il}
\thanks{$^1$This research was supported by the Israel Science Foundation (grant No.\ 152/13)}

\address{Andrei S.~Rapinchuk\\
         Department of Mathematics\\
         University of Virginia\\
         Charlottesville, Virginia 22904\\
         USA}
\email{asr3x@weyl.math.virginia.edu}
\thanks{$^2$Partially supported by NSF grant DMS-1301800 and BSF grant 201049.}

\address{Yoav Segev \\
         Department of Mathematics \\
         Ben-Gurion University of the Negev\\
         Beer-Sheva 84105 \\
         Israel}
\email{yoavs@math.bgu.ac.il}

\begin{abstract}
In the last two decades new techniques emerged to construct valuations
on an infinite division ring $D,$ given a normal subgroup $N\subseteq \Dt$ of finite
index.  These techniques were based on the commuting graph of $\Dt/N$ in
the case where $D$ is non-commutative,
and on the Milnor $K$-graph on $\Dt/N,$ in the
case where $D$ is commutative.  In this paper we unify these
two approaches and consider V-graphs on $\Dt/N$ and how they lead
to valuations.  We furthermore
generalize previous results to situations of finitely many valuations.
\end{abstract}

\date{\today}

\maketitle

\tableofcontents

\section{Introduction}

Let $D$ be a division ring, and $D^\times$ be its multiplicative group.
We recall that a {\it valuation} on $D$ is a surjective
homomorphism $v \colon D^\times \to \Gamma$ to a {\it totally ordered}
group
$\Gamma$ such that
\[
v(a + b) \ge \min(v(a) , v(b)) \ \ \text{ for all} \ \ a , b \in D^\times, \ b \neq -a.
\]
It is well-known that the presence of a non-trivial valuation (or a suitable family of valuations)
can be the key to understanding a wide range of properties of $D$.
Therefore one is interested in conditions that guarantee the existence of a valuation on $D$ with nice properties.
In the case where $D$ is a commutative field, this was the focus of an extensive study, and several methods
to detect valuations were developed, notably the {\it rigidity} method (see \S\ref{Milnor K-graphs}).

In the non-commutative case a method to construct non-trivial valuations had emerged in
\cite{Seg2}, \cite{RS} and \cite{RSS}.
There a valuation $v$ on $D$ was constructed from the assumption
that $D^\times$ possesses a finite index normal subgroup $N$ with the
quotient $D^{\times}/N$ having certain properties.
In particular the diameter of the {\it commuting graph} of
this quotient should not be too small (see \S\ref{Commuting graphs}).
Then, in addition, the subgroup $N$ turns out to be open in the topology on $D$ defined by $v$.
Results of this nature were shown to be very useful in the analysis of the
normal subgroup structure of $D^\times$, when $D$ is a finite-dimensional division algebra.

Next, in \cite{Efr4} it was shown that the method of \cite{Seg2}, \cite{RS} and \cite{RSS} can be
used also when $D$ is a commutative field, once the above commuting graph (which is trivial in the commutative case) is replaced by a graph related
to the Milnor $K$-ring of $D$ modulo the subgroup $N$ of $D^\times$  (see \S\ref{Milnor K-graphs}).

The purpose of this paper is to unify and generalize these two constructions,
by axiomatizing this new approach to the construction of valuations on division rings, whether commutative or not.
This axiomatization leads us to the notion of a {\it valuation graph
associated with} a finite index normal subgroup $N$ of $\Dt$ (see subsection \ref{sub vg}).
This notion, in turn, leads to a uniform approach for constructing maps $\gvp\colon N\to\gC$ having certain
properties resembling those of a valuation on $D,$ where $\gC$ is a {\it partially ordered group}
(see Theorem B  below).
Once such a map $\gvp$ is obtained machinery from \cite{RS, Efr4} can be used to construct
a valuation
(see Theorem A below).  Further, using $\gvp$ and certain additional hypotheses,
and expanding on machinery from \cite{RSS, R} leads us to
new openness results with respect to a {\it finite set of valuations} (and not a single valuation) for $N$.
This is done in Theorem \ref{thm val2} which is then applied to obtain Theorem C.

We hope that our notion of a valuation graph will facilitate future applications of our methods.

We refer to e.g., \cite{Efr_book}, \cite{Endler}, \cite{EnglerPrestel}, \cite{MMU}, \cite{Riben}, \cite{Schil}, \cite{Wad}
for general facts and notions in valuation theory.  See also \S8.

\clearpage

\subsection{Valuations via valuation graphs}\label{sub vg}\hfill
\medskip

\noindent
In what follows let $D$ be an infinite division ring and let $N$ be a proper finite index normal subgroup of
$D^\times$ containing $-1$.
Given $a\in D^\times$ let $a^*=aN$ be the corresponding coset in  $D^\times/N$.

We consider undirected graphs $\Delta$ whose vertices are the non-identity elements of the quotient $D^\times/N$.
We denote the distance function on the vertices of $\Delta$ by $d(\cdot,\cdot)$.
We set $d(a^*,b^*)=\infty$ if the vertices $a^*,b^*$ are not on the same connected component of $\Delta$.
Also, let $\diam(\Delta)$ be the {\it diameter} of the graph $\Delta$, i.e., $ \diam(\Delta)=\sup d(a^*,b^*)$, with $a^*,b^*$ ranging over all vertices.

\begin{Defi}\label{def V-graph}
We say that $\Delta$ as above is a V-{\it graph} (or a {\it valuation graph}) for $D$ if
for every $a , b, c \in D^\times \sminus N$ the following three conditions hold:

\vskip2mm

(V1) if $a - b \in N$ then $d(a^* , b^*) \leqslant 1$,

\vskip1mm

(V2) if $d(a^* , b^*)\leqslant 1$ then $d((a^{-1})^* , b^*) \leqslant 1$,

\vskip1mm

(V3) if $ab \not\in N$ and both $d(a^*,c^*)\leq 1$ and $d(a^*b^*,c^*)\le 1$, then also $d(b^*,c^*) \le 1$.
\medskip

\noindent
\end{Defi}
We then say that the V-graph\ $\Delta$ is {\it associated} with the normal subgroup
$N$, or that the quotient $D^\times/N$ {\it supports} the V-graph $\Delta$.

\vskip2mm

%
\begin{remarks*}
(1) \quad
Axiom (V1) is equivalent to the representative-free condition

\vskip1mm

(V$1'$) if $1 \in a^* + b^*$ then $d(a^* , b^*) \leqslant 1$.
\vskip1mm

\noindent
Here $a^*+b^*$ is the set of all sums of an element of $a^*$ and an element of $b^*$.
The condition $1\in a^*+b^*$ is precisely the {\it Steinberg relation} in the relative
version of Milnor's $K$-theory (which will be described in \S\ref{Milnor K-graphs}).

\medskip

(2)\quad
In sections 4, 5 and 6 we actually use the following weaker
\vskip1mm

(V$3'$)  if  $d(a^* , (ab)^*) \le 2$ or $d(a^* , (ba)^*) \le 2$, then  $d(a^* , b^*) \le 2$,
\vskip1mm

\noindent
in place of (V3) (see Remarks \ref{rem a-b in N}).

%
%





%

\medskip

(3) \quad
Axiom (V3) is equivalent to the following axiom:
\vskip1mm

(V$3''$)\quad
For all $a, b, c\in\Dt\sminus N,$ with $a^*\ne b^*,$ if $d(a^*, c^*)\le 1$ and  $d(b^*, c^*)\le 1,$ then $d((a^*)^{-1}b^*, c^*)\le 1.$
\vskip1mm

Therefore axioms (V2) and (V3) just mean that, for every $c\in\Dt\sminus N$, the set
\[
\{a^*\in\Dt/N\sminus\{1^*\}\mid d(a^*, c^*)\le 1\}\cup\{1^*\}.
\]
is a subgroup of $N$.
\end{remarks*}

\vskip2mm

One of our results is the following.

\vskip1mm

\begin{thmA}
Let $D$ be an infinite division ring and let $N\subseteq \Dt$ be
finite index normal subgroup containing $-1$.
Assume that $\Dt/N$ supports a V-graph $\gD$ of diameter $\ge 4$.
Then in each of the following situations:
\begin{itemize}
\item[(i)]
$D$ is commutative,

\item[(ii)]
$D$ is a finite-dimensional division algebra over a field of finite transcendence degree over its prime field,
\end{itemize}
\noindent there exists a non-trivial valuation $v$ of $D$ such that
$N$ is open in $D$ with respect to the topology defined by $v$.
\end{thmA}
\noindent
Theorem A is proved at the very end of \S10.
One can also think about the openness of the subgroup $N$ of Theorem A as a {\it ``congruence subgroup property''},
turning Theorem A into a {\it ``congruence subgroup theorem''} for finite index normal subgroups $N$ of $D^\times$
such that the quotient $D^\times/N$ supports a V-graph\ of diameter $\ge 4$.

It is important to point out that the lower bound of $\ge 4$ on $\diam(\Delta)$ in Theorem A is optimal.
Namely,  there are examples where the quotient $D^\times/N$ supports a
V-graph of diameter $3$, but $N$ is not open with respect to any non-trivial
valuation (cf.~\cite[Example 8.4]{RS}, \cite[Example 7.2]{Efr4}, and Examples 11.4 and 11.7).


As noted above,
expanding on techniques described earlier in \cite{RSS, R}
we prove a ``congruence subgroup property'' also in the case
where $D^\times/N$ supports a
V-graph of diameter $3,$ but then we require an additional hypothesis
(see Theorem C and the paragraphs following it).  We do
not know if this hypothesis could be removed and we ask:

\begin{question}
Let $D$ be an infinite division ring and let $N\subseteq \Dt$ be
finite index normal subgroup containing $-1$.
Assume that $\Dt/N$ supports a V-graph $\gD$ of diameter $\ge 3,$
and that one of the following holds:
\begin{itemize}
\item[(i)]
$D$ is commutative.

\item[(ii)]
$D$ is a finite-dimensional division algebra over a
finitely generated field.
\end{itemize}
\noindent Does there exist a non-empty finite set $\widetilde{T}$ of  non-trivial valuations of $D$ such that
$N$ is open in $D$ with respect to the topology defined by $\widetilde{T}$?
\end{question}

\noindent
A positive answer to Question 2 will have various applications:
it will enable one to deduce the existence of valuations
in more general situations; it will restrict, in some cases,
the structure of $\Dt/N$, and furthermore,
to complete the proof of  the main result of \cite{RSS} (see subsection \ref{sub sub com} ahead),
one could use the fact that the diameter
of the commuting graph of minimal non-solvable groups is $\ge 3$
(this is shown in \cite{Seg1}), and not that such groups
have property \ETH\ (a much harder task done in \cite{RSS}).
Question 2 was asked in \cite[Question 1, p.~932]{RSS}, in the
case where $\gD$ is the commuting graph of $\Dt/N$.
\vskip2mm

\subsection{Valuation-like and leveled maps}\label{sub vlm}\hfill
\medskip

\noindent
The construction
of a valuation from a V-graph\ $\Delta$ whose vertices are the non-identity
elements of $D^\times/N$, is a {\it two-step procedure}:
\medskip

\noindent
{\bf Step 1.}
Use the axioms of a V-graph\ in conjunction with additional hypotheses, in particular
assumptions on its diameter,
to produce a surjective group homomorphism
$\varphi \colon N \to \Gamma$ to a partially ordered group $\gC,$ with special
properties making it a {\it valuation-like} or {\it (strongly) leveled} map (see the definitions below).
\medskip

\noindent
{\bf Step 2.} Use the maps obtained in {\it step 1}  to construct certain
subrings of $D$ with  properties analogous to those of valuation
rings, and eventually to produce a desired valuation.
\medskip

We would like
to give some indications of the first step as it most directly
relies on the formalism of V-graphs, and does not require any
additional assumptions on $D$.

So let $D$ be an arbitrary infinite division ring and let
$N\subseteq D^\times$ be a finite-index normal subgroup containing $-1$.
Given a partially ordered group $\Gamma$ (written additively, but not necessarily commutative)
and a homomorphism $\varphi \colon N \to \Gamma$, we will
frequently use the following (and similar) notation: for $\alpha\in\Gamma$ we set
\[
\Gamma_{< \alpha} = \{ \beta \in \Gamma \ \vert \ \beta < \alpha
\} \ \ \text{and} \ \ N_{< \alpha} = \{ x \in N \ \vert \ \varphi(x)< \alpha \}.
\]
A homomorphism $\varphi \colon N \to \Gamma$ to a partially
ordered group is said to be a {\it leveled map} 
if there exists a non-negative $\alpha \in \Gamma$ (called a {\it level} of $\varphi$) such that
$N_{<-\ga}\ne\emptyset$ and
\[\tag{L}
N_{< -\alpha} + 1 \subseteq N_{< -\alpha}.
\]
A leveled map to a {\it totally ordered} group $\Gamma$ is called a
{\it valuation-like map}. Next, we say that a homomorphism
$\varphi \colon N \to \Gamma$ to a partially ordered group is a {\it strongly leveled map},
if there
exists a non-negative $\alpha \in \Gamma$ (called a {\it s-level} of
$\varphi$) such that $N_{>\ga}\ne\emptyset$ and
\[\tag{SL}
1 \pm N_{> \alpha} \subseteq N_{\leqslant 0}
\]
(note that while $-1 \in N$ by our assumption, we are {\it not}
assuming that $\varphi(-1) = 0$, which explains the presence of $\pm$).
A strongly leveled map to a {\it totally ordered} group $\Gamma$ is called a
{\it strong valuation-like map}.
We note that by Lemma \ref{lem s-level and level}(2),
{\it a strongly leveled map of s-level $\ga$ is a leveled map of level $\ga$}.
\medskip

One of our main results is:

\begin{thmB}
Let $D$ be an infinite division ring and let $N\subseteq \Dt$ be
finite index normal subgroup containing $-1$.
Assume that $\Dt/N$ supports a V-graph $\gD$.  Then
\begin{enumerate}
\item
if $\diam(\Delta)\ge 3$ then $N$
admits a strongly leveled map;
\smallskip

\item
if $\diam(\Delta)\ge 4$ then $N$
admits a strong valuation-like map;
\smallskip

\item
If $\diam(\Delta)\ge 5$ then $N$
admits a strong valuation-like map of s-level $0$.
\end{enumerate}
\smallskip

\noindent
Furthermore, if $D$ is finite dimensional over a subfield $k\subseteq F=Z(D),$
then in all three cases above $N_{\ge 0}$ contains a basis of $D$ over $k$.
\end{thmB}

\noindent
Parts (1), (2) and (3) of Theorem B are  Theorems \ref{thm existence of s-level in diam 3},
\ref{thm existence of VL in diam 4} and  \ref{thm level 0 in diam ge 5}
respectively.  The last part of Theorem B is Corollary \ref{cor N(a inverse) negative}(3).
\medskip

Using part (1) of Theorem B we prove:

\begin{thmC}
Let $D$ be a finite-dimensional separable\footnotemark (but not
necessarily central)
\footnotetext{This means that the center $F$ of $D$ is a separable
extension of $k$, cf. \cite{Pi}.}
division algebra over an infinite field $k$ of finite transcendence degree over its prime field,
and let $N \subseteq D^{\times}$ be a normal
subgroup of finite index containing $-1$. Assume that
$D^{\times}/N$ supports a V-graph of diameter $\ge 3,$
and let $\varphi
\colon N \to \Gamma$  be the strongly leveled map obtained in  Theorem B(1).

Suppose in addition that
the subgroup $\varphi(N \cap k^{\times}) \subseteq \Gamma$ is
\emph{totally ordered.}
Then
\begin{enumerate}
\item
the restriction $\varphi_k =\varphi \restr_{(N \cap k^{\times})}$ is a strong valuation-like map;

\item
there exists a height one valuation $v$ of $k$ such that $N \cap
k^{\times}$ is open in the $v$-adic topology on $k^\times$;

\item
there exists a non-empty finite set $T$ of valuations of the center $F = Z(D)$ extending $v$ such that
$\vert T\vert\le [F\colon k],$ and such that each $w \in T$ uniquely extends to a valuation $\widetilde{w}$ of $D$,
and $N$ is open in $D^{\times}$ in the $\widetilde{T}$-adic topology,
where $\widetilde{T} = \{\widetilde{w} \: \vert \: w \in T \}$.
\end{enumerate}
\end{thmC}
We mention that the hypothesis in Theorem C that $\varphi(N \cap k^{\times}) \subseteq \Gamma$ is
totally ordered is used to obtain part C(1).  Then, by C(1), Theorem \ref{thm val comm}
applies.  In particular, part C(2) follows. In Theorem \ref{thm val comm}
there is no use of the notion of V-graphs.
Next,
Theorem B(1) and the hypothesis that $\varphi(N \cap k^{\times})$ is totally
ordered, give the precise hypotheses of Theorem
\ref{thm val2}, and that theorem proves part C(3). In Theorem \ref{thm val2}
there is no use of the notion of V-graphs.

Theorem C is proved at the end of \S10.  As noted above, we do not know whether
the hypothesis in Theorem C, that $\gvp(N\cap k^{\times})$ is
totally ordered, can be removed.
The examples in  \S 11 show that Theorem C considers
situations which are more general than those considered by Theorem A.
Finally, we draw the attention of the reader to
Theorems \ref{thm 3.5} and \ref{thm v-3.5}.

Going back to V-graphs, notice that for every $D$ and $N$ as above there is a
canonical (minimal) V-graph supported by $D^\times/N$,
namely, the intersection of all V-graphs supported by $D^\times/N$.
In Theorems A--C, the assumptions that $D^\times/N$ supports a V-graph with sufficiently large diameter
can therefore be replaced by the assumption that this canonical V-graph has such a diameter.

\subsection{The origin and examples of V-graphs}\hfill
\medskip

\noindent
The notion of a V-graph\ has two prototypes: the commuting graph
and the Milnor $K$-graph -- the axioms in Definition \ref{def V-graph} simply postulate
the properties that were used to produce valuations in these two cases.
It is quite remarkable that the essential properties turned out to be identical in these two quite different situations,
which we will now review to put our results in perspective.

\vskip1mm

\subsubsection{\bf Commuting graphs}\label{sub sub com}\hfill
\label{Commuting graphs}
\medskip

\noindent
Let $G$ be a finite group.
The {\it commuting graph} $\Delta_G$ is the
undirected graph whose vertex set consists of the non-identity elements of $G$, and where
two vertices are connected by an edge if and only if the
corresponding elements commute in $G$.
Given an infinite division ring $D$ and a normal subgroup $N\subseteq  D^\times$ (not necessarily
of finite index), it is an easy exercise to check that $\Delta_{D^\times/N}$ is a V-graph
associated with the normal subgroup $N$.  Indeed see \cite[Remark 2.2]{Seg2} for (V1),
(V2) is trivial and (V3) is straightforward.

Let $D$ be any {\it finite dimensional} division algebra and $N$ a {\it finite index} normal subgroup of $D^\times$.
Set $\Delta:=\Delta_{D^\times/N}$.
As mentioned in \S\ref{sub vlm}, producing
a valuation on $D$ using the commuting graph $\Delta$ requires {\it two steps}.

The basic machinery for implementing {\it step 1} was developed in \cite{Seg2}.
Subsequently, it was further developed and improved in \cite{RS} and \cite{RSS}.
Cumulatively, the results obtained in \cite{Seg2}, \cite{RS} and \cite{RSS} yield a proof of Theorem B in the case where
$\Delta=\gD_{\Dt/N}$.

In \cite{SS} it was shown (using the classification of finite simple groups (CFSG))
that if $L$ is a non-abelian finite simple group, then either\linebreak
\mbox{$\diam(\Delta_L)\ge 5$},
or $\Delta_L$ is {\it balanced}
(see \cite{Seg2} for the definition of a balanced commuting graph).
This result, together with \cite[Theorem A]{Seg2} proved \cite[Theorem 3, p.~126]{SS},
which states that for  $D$ and $N$ as above, {\it $D^\times/N$ is not a non-abelian finite simple group}.
This last result was conjectured in \cite{RPo}, and in view of the reduction obtained therein, concluded
the proof of the {\it Margulis--Platonov conjecture} (MP) for inner forms of anisotropic groups of
type $A_n,$ i.e.\ groups of the form $\mathrm{SL}_{1 , D}$ where $D$ is finite-dimensional division
algebra over a global field $K$ (see \cite[Ch.\  9]{PlR} and Appendix A in \cite{RS} for a discussion of (MP)).

A systematic use of valuations in this context was introduced in \cite{RS}, although some
features of valuations can already be seen in \cite{Seg2} (like the local
ring constructed in \S 10 of \cite{Seg2} -- see Appendix B in \cite{RS} for a
discussion of this ring in the context of valuations).
Indeed valuation theory together with the machinery developed in \cite{RS}
supplies the tools adequate for handling {\it step 2}.
This, together with improved results for {\it step 1} in \cite{RS},
enabled the second and third-named authors to construct, under the hypothesis that $\diam(\Delta)\ge 4,$
a valuation $v$ on $D$ such that $N$ is $v$-adically open.

The next major development was the result proved in \cite{RSS}, stating that
{\it for any finite-dimensional division algebra $D$ over an
arbitrary field, every finite quotient of the multiplicative group $D^\times$ is solvable}.
This was based on upgrading the techniques both in {\it step 1} and {\it step 2}
in the context of a new {\it property \ETH} for $\Delta_{D^\times/N}$ (see \S 7)
and establishing this property for {\it minimal non-solvable groups}

In \cite{R}, improving and expanding on machinery for {\it step 2} and using the
results on {\it step 1} from \cite{RSS}, it was shown that if the center of $D$ is a global
field, and if $\diam(\Delta)\ge 3,$ then there is a {\it finite set} of valuations on $D$ so that
$N$ is open in the topology on $D$ defined by this set.
This also enabled the proof of (MP) for inner forms of anisotropic groups of type $A_n,$ using the fact
(which relies on CFSG)
that all finite simple groups are generated by two elements.

\subsubsection{\bf Milnor $K$-graphs}\hfill
\label{Milnor K-graphs}
\medskip

\noindent
To describe the second prototype of a V-graph, recall from \cite[Ch.\ 24]{Efr_book} the definition of the
Milnor $K$-groups of a field $F$ relative to a subgroup $N$ of $F^\times$.
For a non-negative integer $r$ let $K^M_r(F)/N$ be the
quotient of the $r$th tensor power
\[
(F^{\times}/N)^{\otimes r} = (F^{\times}/N) \otimes_{\zz} \cdots
\otimes_{\zz} (F^{\times}/N)
\]
by the subgroup generated by all {\it Steinberg elements}, i.e., elementary tensors
$a_1N \otimes \cdots \otimes a_rN$ such that $1 \in a_iN+ a_jN$ for some $1 \leqslant i < j \leqslant r$
(compare with axiom (V$1'$) above).
The tensor product induces on
\[
K^M_*(F)/N:=\bigoplus_{r=0}^\infty K^M_r(F)/N
\]
the structure of a graded ring.
It is called the {\it Milnor $K$-ring of $F$ modulo $N$}.
Equivalently,   $K^M_*(F)/N$ is the quotient of the (classical)  Milnor $K$-ring $K^M_*(F)\,(=K^M_*(F)/\{1\})$
by the graded ideal generated by $N$, considered as a subgroup of $F^\times=K^M_1(F)$.
 Following traditional notation, the image of $a_1N \otimes \cdots \otimes a_rN$ in
$K^M_r(F)/N$ (where $a_1,\ldots, a_r\in F^\times$) will be denoted by $\{a_1, \ldots , a_r\}_N$.

Now one defines the {\it Milnor $K$-graph of $F$ modulo $N$} to be the undirected graph whose vertices are
the non-identity elements of $F^\times/N$, and where vertices $aN$
and $bN$ are connected by an edge if and only if $\{ a , b \}_N = 0$ in $K^M_2(F)/N$.
It follows from \cite[Lemma 2.1]{Efr4} that this is indeed a V-graph.
The main result of \cite{Efr4} is just Theorem A for this V-graph on $D=F$.

We point out that connections between existence of arithmetically-interesting valuations on a (commutative) field $F$
and Milnor $K$-theory were noted before.
Notably, a series of works by Ware \cite{War},  Arason, Elman, Jacob, and Hwang  (\cite{Jacob1}, \cite{Jacob2}, \cite{AEJ}, \cite{HwangJacob})
developed a method to produce valuations on $F$ using so-called {\it rigid subgroups} of $F^\times$.
In \cite{Efrat_construction} it was shown that this method can be naturally interpreted in terms of relative
Milnor $K$-theory.
In fact, this was one of the main motivations for introducing the relative Milnor $K$-ring functor $K^M_*(F)/N$.
This new perspective opened the way to further strengthening of the rigidity method for producing valuations in \cite{Efrat_TAMS} and \cite[Ch.\ 26]{Efr_book},  and recently in \cite{Topaz1} and \cite{Topaz2}.
Another powerful approach for the detection of valuations on fields related to Milnor $K$-theory was developed by
Bogomolov and Tschinkel  (see e.g., \cite{BT1}, \cite{BT2}, \cite{BT3}).

For some other approaches for the construction of valuations on fields see  \cite{Koenig} and \cite{Becker}.

\vskip2mm

As we see, the situations where the commuting graphs and the Milnor
$K$-graphs were used to construct valuations are indeed quite different (in fact, disjoint), while the results and the techniques involved in their proofs are very much parallel.
This observation
led us to generalize and axiomatize these considerations which resulted in the notion of a V-graph.
\vskip2mm


%
%
\section{Partially preordered and ordered groups}\label{sect partially ordered}\label{sect pog}
%
The goal of this paper is to construct {\it valuations} on an infinite,
finite dimensional division algebra $D,$ given a finite index
normal subgroup $N\subseteq \Dt$ such that the quotients $\Dt/N$ supports
a V-graph.
As indicated in subsection \ref{sub vlm} of the introduction,
this process is carried out in two main steps.  The purpose of this
section is to give more details about {\it step 1}.
Since Step 1 leads to partially preordered and ordered groups,
we discuss in this section such groups in more detail.

So for $x\in\Dt,$ let $x^*$ denote its image in $\Dt/N$.
Now {\it step 1} is achieved using the following further steps:
\medskip

\noindent
{\bf Step 1a}.
Given $y\in \Dt$ we define an {\it invariant binary relation} $\frakP_{y^*}$ on
$N$ such that $(N, \frakP_{y^*})$ is a {\it partially preordered group}.
The relation $\frakP_{y^*}$ will only depend on the coset $y^*=yN$
{\it and not on the coset representative $y$}. This step {\it does not} require $\Dt/N$ to support a V-graph.
Thus below we define and discuss all notions relevant to step 1a.
\medskip

\noindent
{\bf Step 1b}.
Given the partially preordered group $(N, \frakP_{y^*})$
of step 1a we define
\[
U_{y^*}:=\{n\in N\mid 1\,\frakP_{y^*}\, n\text{ and }n\,\frakP_{y^*}\, 1\}.
\]
We show that $U_{y^*}\nsg N$ and that $\gC_{y^*}:=N/U_{y^*}$ is a partially ordered group.
The order relation $\le_{y^*}$ on $\gC_{y^*}$ is given by $mU_{y^*}\le_{y^*} nU_{y^*}$
iff $m\,\frakP_{y^*}\, n,$ where $m,n \in N$.  We let
\[
\gvp_{y^*}\colon N\to \gC_{y^*},
\]
be the canonical homomorphism.
This step as well {\it does not} require $\Dt/N$ to support a V-graph.
Thus below we also discuss all
notions relevant to step 1b.
\medskip

\noindent
{\bf Step 1c}.
We show that if $-1\in N$ and $\Dt/N$ supports a V-graph $\gD,$
then the assertions of Theorem B of the introduction hold,
where the asserted map in parts (1)-- (3) of Theorem B
is $\gvp_{y^*}$, for an appropriate $y^*$.

\begin{remark}
Step 1, and all its parts above {\it do not} require
that $D$ be finite dimensional.  It is only in Step 2,
when we construct valuations on $D,$ that we assume that
$D$ is finite dimensional.
\end{remark}
\medskip

\noindent
{\bf Partially preordered and ordered groups}\label{sub ppog}\hfill
\medskip

\noindent
Let $\gC$ be a group and let $\leq$ be a partial order on $\gC$.
We say that $(\gC,\leq)$ is a \textit{partially ordered group} if
\[
\ga\leq\gc \text{ and } \gb\leq\gd\ \implies\ \ga+\gb\leq\gc+\gd\, ,
\]
for all $\ga,\gb,\gc,\gd\in\gC$.
Notice that we always use \textit{additive notation} for a partially ordered group $(\gC,\leq)$,
even though $\gC$ need not be commutative.
As before for $\alpha\in\gC$ we write  $\gC_{>\alpha}=\{\beta\in\gC\ |\ \beta>\alpha\}$.

A partial order $\leq$ on $\gC$ is called \textit{trivial} if $\alpha\leq\beta\Longleftrightarrow\alpha=\beta$,
for all $\alpha,\beta\in\gC$.


Next let $N$ be an arbitrary group.
We study pullbacks of partial orderings on groups to $N$.

We say that a binary relation  $\frakP$ on $N$ is \textit{invariant} if
\begin{equation}\label{eq invariant}
m\,\frakP\, n\ \implies\ sm\,\frakP\, sn \text{ and } ms\,\frakP\, ns,
\end{equation}
for all $m,n,s\in N$.

\begin{lemma}\label{lemma on invariant relations}
Let $\frakP$ be a reflexive and transitive binary relation on $N$.
Then $\frakP$ is invariant if and only if
\[
m\,\frakP\, n\text{ and } s\,\frakP\, t \implies ms\,\frakP\, nt,\quad \forall m,n,s,t\in N.
\]

\end{lemma}

\begin{proof}
The ``if'' part is immediate.

For the ``only if'' part let $m,n,s,t\in N$ and suppose that $m\,\frakP\, n$ and $s\,\frakP\, t$.
By the invariance, $ms\,\frakP\, ns\,\frakP\, nt$, so by the transitivity, $ms\,\frakP\, nt$.
\end{proof}

When the conditions of Lemma \ref{lemma on invariant relations} are satisfied
we will say that $(N,\frakP)$ is a \textit{partially preordered  group}.

\begin{lemma}\label{lemma on pullbacks of POGs}
The following conditions on a  binary relation $\frakP$ on $N$ are equivalent:
\begin{enumerate}
\item[(1)]
$\frakP$ is reflexive, transitive and invariant (i.e., $(N,\frakP)$ is a partially preordered group).
\item[(2)]
There exist a partially ordered group $(\gC,\leq)$ and a group-epimorphism
$\varphi\colon N\to\gC$ such that $m\,\frakP\, n\Longleftrightarrow \varphi(m)\leq\varphi(n)$ for all $m,n\in N$.
\end{enumerate}
Moreover, when these conditions are satisfied, the kernel of $\varphi$ is
\[
U=\{n\in N\ |\ 1\,\frakP\, n\, \text{ and }n\,\frakP\, 1\}.
\]
\end{lemma}

\begin{proof}
(1)$\Rightarrow$(2): \quad
We first show that $U$ is a subgroup of $N$.
By the reflexivity, $1\in U$.
If $m,n\in U$, then by Lemma \ref{lemma on invariant relations},
$mn\,\frakP\,1\cdot1$ and $1\cdot1\,\frakP\, mn$, so $mn\in U$.
Also,  the invariance gives $m^{-1} m\,\frakP\, m^{-1}\cdot 1$ and
$1\cdot m^{-1}\,\frakP\, mm^{-1}$, so $m^{-1}\in U$, as desired.

Next we observe that $U$ is normal in $N$.
Indeed, let $m\in U$ and $n\in U$.
Then, by the invariance, $n^{-1} mn\,\frakP\, n^{-1}\cdot1\cdot n$ and
$n^{-1}\cdot 1\cdot n\, \frakP\, n^{-1} mn$,
so $n^{-1} mn\in U$.

We further notice that the relation $m\, \frakP\, n$ depends only on the cosets of $m$ and $n$ modulo $U$.
Indeed let $u,v\in U$.
Then $u\,\frakP\,1$ and $1\,\frakP\, v$.
Therefore $m\,\frakP\, n$ implies (by the invariance) that $mu\cdot1\,\frakP\, n\cdot1\cdot v$.

Now set $\gC=N/U$ and let $\varphi\colon N\to\gC$ be the canonical epimorphism.
By what we have just seen, we may define a binary relation $\leq$ on $\gC$ by
$\varphi(m)\leq\varphi(n)\iff m\,\frakP\, n$ for $m,n\in N$.

Since $\frakP$ is reflexive and transitive, so is $\leq$.
Also, if $m,n\in N$ and $\varphi(m)\leq\varphi(n)\leq\varphi(m)$, then $m\,\frakP\, n\frakP\, m$.
Multiplying by $m^{-1}$ on the left, we see that $1\,\frakP\, m^{-1} n\frakP\, 1$,
so $m^{-1} n\in U$, whence $\varphi(m)=\varphi(n)$.
Thus $\leq$ is a partial order.

Finally, the fact that $(\gC,\leq)$ is a partially ordered \textit{group} follows from
Lemma \ref{lemma on invariant relations}.

\medskip

(2)$\Rightarrow$(1): \quad
Straightforward.
\end{proof}

\begin{remark}
Notice that in Lemma \ref{lemma on pullbacks of POGs}, and throughout this article,
we use multiplicative notation for $N$ and additive notation for $\gC$.
\end{remark}

\section{The group $\gC_{y^*}$ and the map $\gvp_{y^*}\colon N\to\gC_{y^*}$}\label{sect the ordered gp}
In this section $D$ is an arbitrary infinite division algebra
(not necessarily finite dimensional over its center),
and $N$ is a normal subgroup  of $\Dt$ of finite index.
Note that in this section we make {\it no additional hypotheses}.
In particular, {\it we do not assume that $\Dt/N$ supports} a V-graph.
As before, for $a\in \Dt$ denote by $a^*$ the image of $a$ in
$\Dt/N$ under the canonical homomorphism.

Our goal in this section is, using only the above information, to construct
for any $y\in \Dt\sminus N$ a map
\[
\gvp_{y^*}\colon N\to\gC_{y^*},
\]
where here $\gC_{y^*}$ is a partially ordered group.
\medskip

We start with defining a
binary relation $\frakP_{y^*}$ on $N$
by
\[
m\,\frakP_{y^*}\, n\quad \Longleftrightarrow\quad N(my)\subseteq N(ny).
\]
A crucial role is played by the sets $N(y)$:
for $y\in\Dt$ we let
\[
N(y):= \{n\in N\mid y+n\in N\}=N\cap (N-y).
\]
Lemma \ref{lem basic properties of N(y)}
below gives some basic properties of the
sets $N(y)$.  Then Corollary \ref{cor on inclusion of N sets}(2) shows that
$\frakP_{y^*}$ depends only on the coset $y^* =Ny$ and not
on the coset representative $y$.  Furthermore, Corollary \ref{cor on inclusion of N sets}(1) shows that $\frakP_{y^*}$
does not depend on
the ``side'', i.e., the relation defined by $N(ym)\subseteq N(yn)$
coincides with $\frakP_{y^*}$.
In Lemma \ref{lem frakP is a preorder relation} we see that
$(N,\frakP_{y^*})$ is a partially preordered group,
and then we use Lemma \ref{lemma on pullbacks of POGs}
to define the partially ordered group $(\gC_{y^*}, \le_{y^*})$ and
the map $\gvp_{y^*}$.

\begin{lemma}[\cite{RS}, Lemma 6.3]\label{lem basic properties of N(y)}
Let $y\in \Dt\sminus N$ and $n\in N$.  Then
\begin{enumerate}
\item $N(ny)=nN(y)$ and $N(yn)=N(y)n$;

\item $N(y^x)=x^{-1}N(y)x,$ for all $x\in \Dt$;

\item $N(y)\ne\emptyset$;

\item if $n\in N(y^{-1})$, then $y+n^{-1}\in Ny$.
Consequently, $n^{-1}\notin N(y)$.  In particular
$\emptyset \subsetneqq N(y)\subsetneqq N$.
\end{enumerate}
\end{lemma}
\begin{proof}
(1):\quad
$N(ny)=N\cap (N-ny)=nN\cap(nN-ny)=n(N\cap(N-y))=nN(y)$,
and similarly for $N(yn)$.
\medskip

\noindent
(2):\quad
$N(y^x)=N\cap(N-y^x)=N^x\cap(N^x-y^x)=N(y)^x$.
\medskip

\noindent
(3):\quad
This is an immediate consequence of the fact that $D= N - N$ (cf. \cite{BSh}, \cite{Tu}).
\medskip

(4):\quad
The first part of (4) follows from the definition of $N(y)$ and
the rest of (4) is a consequence of the first part, and of (3).
\end{proof}

\begin{cor}\label{cor on inclusion of N sets}
Let $y\in D^\times\sminus N$ and $m,n\in N$.
Then
\begin{enumerate}
\item[(1)]
$N(my)\subseteq N(ny)$ if and only if $N(ym)\subseteq N(yn)$;
\item[(2)]
if $N(my)\subseteq N(ny)$, then $N(my')\subseteq N(ny')$ for all $y'\in yN=Ny$.
\end{enumerate}
\end{cor}
\begin{proof}
(1): \quad
The first inclusion is equivalent to $yN(my)y^{-1}\subseteq yN(ny)y^{-1}$,
which by Lemma \ref{lem basic properties of N(y)}(2) is just the second inclusion.
\medskip

\noindent
(2):\quad
Multiply the inclusion $N(my)\subseteq N(ny)$
on the right by elements of $N$ and use Lemma \ref{lem basic properties of N(y)}(1).
\end{proof}

We now show that $(N, \frakP_{y^*})$ is a partially preordered group, for any $y\in \Dt\sminus N$.

\begin{lemma}[\cite{RS}, Lemma 6.4]\label{lem frakP is a preorder relation}
For any $y\in \Dt\sminus N$, the relation $\frakP:= \frakP_{y^*}$
has the following properties
\begin{enumerate}
\item $\frakP$ is reflexive and transitive;

\item
$\frakP$ is invariant.
%
%
\end{enumerate}
\end{lemma}
\begin{proof}
(1):\quad
This is immediate from the definition of $\frakP$.

(2):\quad
Recall from equation \eqref{eq invariant}
the notion of an invariant relation.  Let $m, n, s\in N,$ and assume that $m\,\frakP\,n$.
By the definition of $\frakP$ and Lemma \ref{lem basic properties of N(y)}(1),
$sm\,\frakP\, sn$.
Next  we have $N(my')\subseteq N(ny'),$ for all $y'\in Ny,$
by Corollary \ref{cor on inclusion of N sets}(2).
Taking  $y'=sy,$ we see that $N(msy)\subseteq N(nsy)$, i.e.,
$ms\,\frakP\, ns$.
%
%
\end{proof}

Since $(N,\frakP_{y^*})$ is a partially preordered group, Lemma \ref{lemma on pullbacks of POGs}
yields a partially ordered group $(\Gamma_{y^*},\leq_{y^*})$ and a group epimorphism
$\varphi_{y^*}\colon N\to\Gamma_{y^*}$ such that
%
\begin{equation}\label{eq meaning of le}
\varphi_{y^*}(m)\leq_{y^*}\varphi_{y^*}(n) \quad\iff \quad m\,\frakP_{y^*}\,n\quad\iff N(my)\subseteq N(ny),
\end{equation}
for all $m,n\in N$.
More concretely,
\[
\Gamma_{y^*}=N/U_{y^*},
\]
where $U_{y^*}$ is the normal subgroup
\[
U_{y^*}:=\{n\in N\mid n\,\frakP_{y^*}\,1\text{ and }1\,\frakP_{y^*}\,n\}=\{n\in N\mid N(ny)=N(y)\}
\]
of $N$, and
\[
\varphi_{y^*}\colon N\to N/U_{y^*},
\]
is the canonical homomorphism.

Next, we let
\[
\mathbb{P}_{y^*}=\{b\in Ny \mid 1\in N(b)\}.
\]
Note that it follows from Lemma \ref{lem basic properties of N(y)}(3) and Lemma \ref{lem basic properties of N(y)}(1) that
$\mathbb{P}_{y^*}\ne\emptyset$.

One has $b\in \mathbb{P}_{y^*}$ if and only if $yb^{-1}\in N$ and $yb^{-1}\in Nyb^{-1}-y=N-y$, or equivalently, $yb^{-1}\in N(y)$.
Therefore
\[
\mathbb{P}_{y^*}=N(y)^{-1}y,
\]
and similarly,
\[
\mathbb{P}_{y^*}=yN(y)^{-1}.
\]

\begin{lemma}\label{lem eq cond}
Let $y\in\Dt\sminus N$. The following conditions on $m,n\in N$ are equivalent:
\begin{enumerate}
\item[(1)]
$m\,\frakP_{y^*}\,n$.
\item[(2)]
$\gvp_{y^*}(m)\le_{y^*}\gvp_{y^*}(n)$.
\item[(3)]
$N(my')\subseteq N(ny')$  for all $y'\in Ny$.
\item[(4)]
$m\in N(nb)$ for all $b\in\pp_{y^*}$.
\item[(5)]
$nb+m\in N$ for all $b\in\pp_{y^*}$.
\item[(6)]
$n\pp_{y^*}\subseteq m\pp_{y^*}$.
\item[(7)]
For all $y'\in Ny$, if $n\in N(y')$, then $m\in N(y')$.
\end{enumerate}
\end{lemma}
\begin{proof}
To simplify notation we denote  $\pp:=\pp_{y^*}$.
\medskip

\noindent
(1)$\Leftrightarrow$(2): \quad
This is immediate from the definition of $\gvp_{y^*}$ and of $\le_{y^*}$.
\medskip

\noindent
(1)$\Rightarrow$(3): \quad
This is Corollary 3.2(2).
\medskip

\noindent
(3)$\Rightarrow$(4): \quad
Let $b\in\pp$.
As $\pp\subseteq Ny$ we have  $N(mb)\subseteq N(nb)$, by (3).
But $m\in N(mb)$, so $m\in N(nb)$
\medskip

\noindent
(4)$\Rightarrow$(1): \quad
We need to show that (4) implies $mN(y)\subseteq N(ny)$.
So take $s\in N(y)$.
Then $b:=ys^{-1}\in\pp$.
By (4), $m\in N(nb)=N(nys^{-1})$ whence $ms\in N(ny)$.
\medskip

\noindent
(4)$\Leftrightarrow$(5): \quad
As $m\in N$, this is immediate from the definition of $N(nb)$.
\medskip

\noindent
(4)$\Leftrightarrow$(6): \quad
Condition (4) is equivalent to $1\in N(m^{-1}nb)$ for all $b\in\pp$.
As $m^{-1}n\pp\subseteq Ny$, this means that $m^{-1}n\pp\subseteq\pp$, as desired.
\medskip

\noindent
(3)$\Rightarrow$(7): \quad
Let $y'\in Ny$.
By (3) (with $y'$ replaced by $n^{-1}y'$), $N(mn^{-1}y')\subseteq N(y')$.
Now if $n\in N(y')$, then $m\in N(mn^{-1}y')$, and it follows that $m\in N(y')$.

\medskip

\noindent
(7)$\Rightarrow$(1): \quad
Let $s\in N(my)$.
Then $1\in N(mys^{-1})$, so $n\in N(y')$, where $y':=mys^{-1}n\in Ny$.
By (7), $m\in N(y')$.
Thus $1\in N(ys^{-1}n)$ and therefore $n^{-1}s\in N(y)$.
We conclude that $s\in nN(y)=N(ny)$.
\end{proof}

To continue the discussion we recall some notation.
Given $y\in \Dt\sminus N$, we let
\[
N_{\le_{y^*}\,\gc}=
\{m\in N\mid \gvp_{y^*}(m) \le \gc\},\qquad\gc\in\gC_{y^*},
\]
and the sets $N_{<_{y^*}\, \gc}$, $N_{>_{y^*}\,\gc}$ etc.
are defined similarly.

\begin{cor}\label{cor properties lf le y}
Let $y\in \Dt\sminus N$.  Then
\begin{enumerate}
\item $N_{\le_{y^*}\, 0}=\{m\in N\mid m\in N(b),
\text{ for all }b\in\pp_{y^*}\}$;

\item for all $y'\in Ny$,
if $n\in N(y'), m\in N$ and $\gvp_{y^*}(m)\le_{y^*} \gvp_{y^*}(n)$, then $m\in N(y')$.
\end{enumerate}
\end{cor}
\begin{proof}
(1):\quad
This follows from the equivalence of (2) and (4) in Lemma \ref{lem eq cond}, with $n=1$.
\medskip

\noindent
(2):\quad
This follows from the implication (2)$\Rightarrow$(7) in Lemma \ref{lem eq cond}.
\end{proof}


The following proposition gives an
important and surprising property
of the sets $N(y),\ y\in\Dt\sminus N$.

\begin{prop}\label{prop N(a) basis}
Assume that $D$ is a finite-dimensional (but not necessarily central)
division algebra over an
infinite field $k$,
and let $N \subseteq D^{\times}$ be a subgroup of
finite index. Then for any $a \in D^{\times} \sminus N$, each of
the sets $N(a)$ and  $N(a)^{-1}:=\{n^{-1}\mid n\in N(a)\}$
contains a basis of $D$ over $k$.
\end{prop}

(Note that this proposition is valid without any additional assumptions on
$D^{\times}/N,$ in the case where $N$ is normal in $D$.)

\begin{proof}
We first recall the following Proposition due to Turnwald.

\begin{prop}[\cite{Tu}, Proposition 1.3]\label{prop tur}
Let $D$ be an
infinite division ring and let $N \subseteq D^{\times}$ be a subgroup of
finite index. Then for any $x_1, \ldots , x_m \in D^{\times}$ there
exists $c \in D^{\times}$ such that $1 + cx_j \in N$ for all $j = 1,
\ldots , m$.
\end{prop}

We now continue with the proof of Proposition \ref{prop N(a) basis}.
Let $x_1, \ldots , x_{\ell}$ be a left transversal for $N$ in
$D^{\times}$.

We claim that there is a basis $y_1,\ldots , y_m$ of $D$ over $k$ contained in $N$.
Indeed, let $H = \mathrm{GL}_{1,D}$ be the algebraic group associated with $D$.
Then $H = \bigcup x_i\widebar{N}$, where $\widebar{N}$ denotes the Zariski-closure of $N$ in $H$.
Since $H$ is connected, we conclude that $H=\widebar{N}$.
Then $N$ is also Zariski-dense in $D$.
The existence of a required basis now follows from the fact that the
$m$-tuples $(y_1, \ldots , y_m) \in D^m$ that constitute a basis of
$D$ over $k$ form a Zariski-open subset.

Now Consider the finite set of elements
\begin{equation}\label{eq basis}
x_i y^{-1}_j \ \ \text{where} \ \ i = 1, \ldots , \ell; \ j = 1,
\ldots , m.
\end{equation}
By Proposition \ref{prop tur}, there exists $c \in
D^{\times}$ such that
\[
1 + c x_i y^{-1}_j \in N \ \ \text{for all} \ \ i, j.
\]
Since $x_1, \ldots , x_{\ell}$ is a left transversal for $N$ in $\Dt,$ there exists
$i_0 \in \{1, \ldots , \ell\}$ and $s\in N,$ such that $c^{-1}a=x_{i_0}s^{-1}$,
that is $cx_{i_0} = as$.  Then
\[
1 + as y^{-1}_j \in N \ \ \text{for all} \ \ j = 1, \ldots , m.
\]
It follows that the elements $y_1 s^{-1}, \ldots , y_m s^{-1}$,
which form a basis of $D$ over $k$, are all contained in $N(a)$. To
show that $N(a)^{-1}$ contains a basis, we apply the same argument
to the family
\[
x_i y_j \ \ \text{where} \ \ i = 1, \ldots , \ell; \ j = 1,
\ldots , m
\]
in place of the family in equation \eqref{eq basis},
to obtain that $sy_1,\ldots,sy_m$ is a basis of $D$ over $k$ contained in  $N(a)^{-1}$.
\end{proof}

\section{The case where ${\rm diam}(\gD)\ge 3$}\label{sect diam 3}
In this section we continue the notation and hypotheses of \S\ref{sect the ordered gp}.
In addition we assume that $-1\in N$ and that $\Dt/N$ supports a V-graph $\gD$.
We write $\gD$ also for the vertex set of $\gD,$ i.e., for the
non-identity elements of $\Dt/N$. The purpose of this section is to
prove Theorem B(1) of the Introduction:
%
%
\begin{thm}\label{thm existence of s-level in diam 3}
Assume that there are elements $x^*, y^*\in\gD$ such that $d(x^*, y^*)\ge 3$.
Then
the map $\gvp_{y^*}\colon N\to \gC_{y^*}$ is a strongly leveled map.
\end{thm}

\noindent
We start with

\begin{remarks}\label{rem a-b in N}
(1) \quad
Note that in Theorem B of the introduction we are assuming that $-1\in N$.
Hence for all $a\in \Dt$, $(-a)^*=a^*$.  Also,
by property (V1),
for every $x,y\in D^\times\sminus N$ and $n\in N$,
if $x+y\in N$ or $x-y\in N$, then $d(x^*,y^*)\le 1$.
Similarly, if $n\notin N(y)$, then $d(y^*,(y+n)^*)\le 1$.
We use these facts without further reference.
\medskip

(2)\quad
Axiom (V3), in conjunction with axiom (V2), also implies:

\medskip

($*$) \quad
If  $d(a^*,c^*)\leq 1$ and $d(b^*a^*,c^*)\le 1$, then $d(b^*,c^*) \le 1$,

\medskip
\noindent
for all $a,b,c\in D^\times\sminus N$ such that $ba\not\in N$.
Indeed,  $d((a^{-1})^*,c^*)\le1$ and $d((a^{-1})^*(b^{-1})^*,c^*)\le1$,
so by (V3), $d((b^{-1})^*,c^*)\le1$, whence $d(b^*,c^*)\le1$.

\medskip
(3) \quad
Axiom (V3) and ($*$) have the following immediate consequence:

\medskip

(V$3'$) \quad
If  $d(a^* , (ab)^*) \le 2$ or $d(a^* , (ba)^*) \le 2$, then  $d(a^* , b^*) \le 2$,

\medskip
\noindent
for all $a,b\in D^\times\sminus N$ such that $ab \notin N$.
In fact, in this section, as well as in sections 5 and 6 (which are based on the results of this section)
we will need axiom (V3) only in this weaker form (V3$'$).
The full strength of axiom (V3) will be needed only in section 7.
\medskip

(4)\quad
Notice that for $a, b\in D\sminus N,$ if $a^*\ne b^*\ne (a^*)^{-1},$ then $d(a^*,b^*)=d((a^*)^{-1}, b^*)$.
Indeed this follows from axiom (V2).

\end{remarks}

The next two lemmas list some  advanced and useful properties
of the sets $N(x)$. We mention that Lemma \ref{lem N(ab)}(1) will be
used only in the next section.

\begin{lemma}[\cite{RS}, Lemma 6.8]\label{lem N(x+y)}
Let $x,y\in \Dt\sminus N,$ and assume that $d(x^*,y^*)\ge 3,$  then,
\begin{enumerate}
\item
$x+y\notin N;$

\item
$N(x+y)=N(x)\cap N(y)$;

\item if $d((x+y)^*,x^*)\ge 3$, then
$N(x+y)=N(y)\subseteq N(x)\cap N(-x)$.
\end{enumerate}
\end{lemma}
\begin{proof}
The proof of this lemma uses only property (V1) of $\gD$.
Part (1) follows from Remark \ref{rem a-b in N}(1).
\medskip

\noindent
(2):\quad
Let $z:=x+y$
and let $n\in N(z)$.
Suppose, say, that $n\notin N(x)$.
We have
\[
(x+n)+y=z+n\in N.
\]
By Remark \ref{rem a-b in N}(1), $d(x^*, (x+n)^*)\le 1\ge d((x+n)^*,y^*)$;
thus $d(x^*, y^*)\le 2$, a contradiction.
This shows that $N(z)\subseteq N(x)$.
Similarly, $N(z)\subseteq N(y)$.

It remains to show that $N(x)\cap N(y)\subseteq N(z)$.
Assume false and
let $n\in (N(x)\cap N(y))\sminus N(z)$.
Then $(z+n)-x=(x+y+n)-x=y+n\in N$.  Similarly, $(z+n)-y\in N$,
so by Remark \ref{rem a-b in N}(1),
$x^*,(z+n)^*,y^*$ is a path in $\Delta$, a contradiction.
\medskip

\noindent
(3):\quad
This follows from (2), since by (2), also
\[
N(y)=N(x+y-x)=N(x+y)\cap N(-x).
\qedhere
\]
\end{proof}

\begin{lemma}[\cite{RS}, Lemma 6.9]\label{lem N(ab)}
Let $a,b\in \Dt\sminus N$ and $\gre\in \{1,-1\}$.  Then,
\begin{itemize}
   \item[(1)]  if $d(a^*,b^*)\ge 4$ and $\gre\notin N(b)$, then
               $N(ab)\cup N(ba)\subseteq N(a)\cap N(-a)$;
   \item[(2)]  if $d(a^*,b^*)\ge 3$ and $\gre\in N(b^{-1})$, then
               $N(ab)\cup N(ba)\subseteq N(a)\cap N(-a)$;
   \item[(3)]  if $d(a^*,b^*)\ge 3$ and $\gre\in N(a)$, then
               $N(b)\subseteq N(ab)\cap N(ba)$.
\end{itemize}
\end{lemma}
\begin{proof}
(1) \& (2):\quad
We first claim that under each of the hypothesis of (1) and the hypothesis of (2)
we have $b+\gre\notin N$ and $d(a^*,(b+\gre)^*)\geq3$.

Indeed, in (1) we assume $d(a^*,b^*)\geq4$ and $b+\gre\notin N$.
Since $d(b^*,(b+\gre)^*)\leq1$ we get $d(a^*,(b+\gre)^*)\geq3$.

In (2) we assume $d(a^*,b^*)\geq3$ and $\gre\in N(b^{-1})$.
By Lemma  \ref{lem basic properties of N(y)}(4), $b+\gre\notin N$ and $(b+\gre)^*=b^*$,
so again $d(a^*,(b+\gre)^*)\geq3$, proving our claim.

From the claim and (V3$'$) we deduce that $ab+\gre a\notin N$ and $d(a^*,(ab+\gre a)^*)\geq3$.
Also, $d(a^*,(ab)^*)\geq3$, by (V3$'$) again.
By Lemma \ref{lem N(x+y)}(3)  (with $\gre a,ab$ in place of $x,y$, respectively),
$N(ab)\subseteq N(\gre a)\cap N(-\gre a)=N(a)\cap N(-a)$.
For the other inclusion conjugate by $a$ using Lemma \ref{lem basic properties of N(y)}(2).

\medskip

\noindent
(3):\quad
Since $d(a^*,b^*)\ge 3$, (V2) and (V3$'$) imply that $d((ab)^*,(a^{-1})^*)\ge 3$,
and it follows from (2) (taking $ab,a^{-1}$ in place of $a,b$, respectively),
that $N(a^{-1}ab)\subseteq N(ab)$, that is $N(b)\subseteq N(ab)$.
The other inclusion is obtained by conjugating by $b$.
\end{proof}
%
%

The next two results are based on \cite{RSS}, Propositions 5.2 and 5.3.
%
%

\begin{prop}\label{prop Px vs Py}
Let $x,y\in D^\times\sminus N$ satisfy $d(x^*,y^*)\geq3$.
Let  $a\in \mathbb{P}_{x^*}$ and $b\in \mathbb{P}_{y^*}$.
Then:
\begin{enumerate}
\item[(1)]
$a+1\in U_{y^*},$ and hence $a+1\in N(b)$;
\item[(2)]
$N(a^{-1})\subseteq N(b)$;
\item[(3)]
$N(a^{-1}b^{-1})\subseteq N(\epsilon a^{-1})\cap N(\epsilon b^{-1})$ for $\epsilon=\pm1.$
\end{enumerate}
Further, let  $n\in N$ satisfy $n^{-1}\in N(a^{-1}b^{-1})$.
Then \begin{enumerate}
\item[(4)]
$N(nb)\neq N(b)$;
\item[(5)]
$(N(a)\cap N(b))\pm n\subseteq N(a)\cap N(b)$;
\item[(6)]
$1\pm n\in N(c)$ for every $c\in\mathbb{P}_{(x^{-1}y^{-1})^*}$.
\end{enumerate}
\end{prop}
\begin{proof}
(1):\quad
First, $a+1\in N,$ because $1\in N(a)$.
By definition, to show that $a+1\in U_{y^*},$
it remains to show that $N(ay+y)=N(y)$.

By (V3$'$) and Remark \ref{rem a-b in N}(2), $d((ay)^*,y^*)=d(x^*y^*,y^*)\geq3$.
Now $(ay+y)^*=y^*,$ since $a+1\in N,$ so $d((ay+y)^*,(ay)^*)=d(y^*,(xy)^*)\geq3$.
By Lemma \ref{lem N(x+y)}(3) (with $ay,y$ in place of $x,y$, respectively), $N(ay+y)=N(y)$.

But now, $\gvp_{y^*}(a+1)=0,$ since $a+1\in U_{y^*},$ so by Corollary
\ref{cor properties lf le y}(1), $a+1\in N(b)$.
\medskip

\noindent
(2):\quad
By (V2),  $d(y^*,(x^{-1})^*)\geq3$.
By Lemma \ref{lem N(ab)}(3),(2) (taking $b,a^{-1}$ in place of $a,b$, respectively, and $\epsilon=1$),
\[
N(a^{-1})\subseteq N(a^{-1}b)\subseteq N(b).
\]

\medskip

\noindent
(3):\quad
Since $d((x^{-1})^*,(y^{-1})^*)\geq3$, we may use Lemma \ref{lem N(ab)}(2),
with $\epsilon a^{-1},\epsilon b^{-1}$ in place of $a,b$, and then with
$\epsilon b^{-1},\epsilon a^{-1}$ in place of $a,b$.
\medskip

\noindent
(4):\quad
As $1\in N(b)$ we have $n\in N(nb)$.
On the other hand, $n^{-1}\in N(b^{-1})$, by (3).
By Lemma \ref{lem basic properties of N(y)}(4), $n\not\in N(b)$.
Therefore $N(nb)\neq N(b)$.
\medskip

\noindent
(5):\quad
By (3), $\pm n^{-1}\in N(a^{-1})\cap N(b^{-1})$, so $a\pm n\in aN$ and $b\pm n\in bN$.
By Lemma \ref{lem N(x+y)}(2) (with $a+n,b-n$ in place of $x,y$), $N(a+b)=N(a+n)\cap N(b-n)$.
Similarly $N(a+b)=N(a-n)\cap N(b+n)$.
Further, by Lemma \ref{lem N(x+y)}(2), $N(a)\cap N(b)=N(a+b)$, so we obtain
\[
N(a)\cap N(b)\subseteq N(a+\epsilon n)\cap N(b+\epsilon n),
\]
for $\epsilon=\pm1$.

Let now $m\in N(a)\cap N(b)$.
We obtain that $a+\epsilon n+m, b+\epsilon n+m\in N$.
Therefore, if $\epsilon n+m\not\in N$, then by Remark \ref{rem a-b in N}(1), $a^*,(\epsilon n+m)^*, b^*$ is a path in $\Delta$,
contradicting $d(a^*,b^*)\geq3$.
Consequently, $\epsilon n+m\in N$, whence $m+\epsilon n\in N(a)\cap N(b)$.
\medskip

\noindent
(6):\quad
By (5), $1\pm n\in N$.
By (1) (with $x^{-1}y^{-1},x,c,a$ in place of $x,y,a,b$) we have $c+1\in N(a)$.
Similarly, $c+1\in N(b)$.
It follows from (5) that $c+1\pm n\in N(a)\cap N(b)\subseteq N$.
Hence $1\pm n\in N(c)$.
\end{proof}

\begin{cor}\label{cor N(a inverse) negative}
Let $x,y\in D^\times\sminus N$ satisfy $d(x^*,y^*)\geq3$.
Then
\begin{enumerate}
\item
$N(a^{-1})\subseteq N_{\le_{y^*}0}$ for every  $a\in \mathbb{P}_{x^*};$

\item
$N_{>_{y^*} \ga}\ne\emptyset$, for all $\ga\in\gC_{y^*};$

\item
if $D$ is a finite-dimensional algebra over an infinite field $k$, then for every $\alpha \in \Gamma_{y^*}$, the set $N_{\geq_{y^*} \alpha}$ contains a basis of $D$ over $k$.
\end{enumerate}
\end{cor}
\begin{proof}
(1):\quad
For every $b\in \mathbb{P}_{y^*}$ Proposition \ref{prop Px vs Py}(2) gives $N(a^{-1})\subseteq N(b)$.
Now use Corollary \ref{cor properties lf le y}(1).
\medskip

\noindent
(2):\quad Let $a, b$ and $n$ be as in Proposition \ref{prop Px vs Py}.
By Proposition \ref{prop Px vs Py}(3), $n^{-1}\in N(a^{-1}),$
so by (1), $n^{-1}\in N_{\le_{y^*} 0}$.  It follows that $n\in N_{\ge_{y^*} 0}$.

Also, by Proposition \ref{prop Px vs Py}(4), $N(nb)\ne N(b),$ so,
by definition, $n\notin U_{y^*}$.  Hence $\gvp_{y^*}(n)\ne 0$.
It follows that $n\in N_{>_{y^*} 0}$.
Since $N_{>_{y^*} 0}\ne\emptyset,$
the assertion now follows from the surjectivity of $\varphi_{y^*}$.
\medskip

\noindent
(3):\quad
Pick $a \in Nx$ so that $1 \in N(a)$.
Then according to Corollary \ref{cor N(a inverse) negative}(1)
we have $N(a^{-1}) \subseteq N_{\leq_{y^*} 0}$, so $N(a^{-1})^{-1} \subseteq N_{\geq_{y^*} 0}$.
Proposition \ref{prop N(a) basis} gives a basis of $D$ over $k$ inside $N_{\geq_{y^*}0}$.

Now take $z\in N$ with $\alpha=\varphi_{y^*}(z)$.
We multiply the above basis by $z$ to obtain a basis of $D$ over $k$ in
 $z N_{\geq_{y^*} 0} = N_{\geq_{y^*} \alpha}$.
\end{proof}

\medskip

\begin{proof}[Proof of Theorem \ref{thm existence of s-level in diam 3}]
Let  $a\in \mathbb{P}_{x^*}$, $b\in\mathbb{P}_{y^*}$ and $n\in N$ with $n^{-1}\in N(a^{-1}b^{-1})$.
Set
\[
z:=x^{-1}y^{-1}.
\]
We claim that $\varphi_{z^*}$ is a strongly leveled map with level $\ga:=\varphi_{z^*}(n)$.
We recall that this means that
\begin{enumerate}
\item[(i)]
$\ga\ge_{z^*}0$ and $N_{>_{z^*}\ga}\ne\emptyset$.
\item[(ii)]
If $m\in N$ and $\varphi_{z^*}(m)>_{z^*}\ga$, then $\varphi_{z^*}(1\pm m)\leq_{z^*}0$.
\end{enumerate}
For (i) we use Corollary \ref{cor N(a inverse) negative} (but with $x,z$ in place of $x,y$).
Indeed by Proposition \ref{prop Px vs Py}(3), $n^{-1}\in N(a^{-1})$,
so by  Corollary \ref{cor N(a inverse) negative}(1),
$\varphi_{z^*}(n^{-1})\leq_{z^*}0$, i.e., $\ga\ge_{z^*}0$.  By
Corollary \ref{cor N(a inverse) negative}(2), $N_{>_{z^*} \ga}\ne\emptyset$.

For $m$ as in (ii) we have $\varphi_{z^*}(m^{-1})<\varphi_{z^*}(n^{-1})$.
By Corollary \ref{cor properties lf le y}(2) (with $z,a^{-1}b^{-1},m^{-1},n^{-1}$ in place of $y,y',m,n$), $m^{-1}\in N(a^{-1}b^{-1})$.
Hence, by Proposition \ref{prop Px vs Py}(6), $1\pm m\in N(c)$ for every $c\in\mathbb{P}_{z^*}$.
We now use Corollary  \ref{cor properties lf le y}(1).

We may now take $x^{-1}, y^{-1}x,y$ in place of $x,y,z$ to conclude that $\varphi_{y^*}$ is also a strongly leveled map.
\end{proof}

\section{The case where ${\rm diam}(\gD)\ge 4$}
In this section we continue with the notation and hypotheses of \S4.
The purpose of this section is to prove Theorem B(2) of the Introduction:
%
%
\begin{thm}\label{thm existence of VL in diam 4}
Assume that there are elements $x^*, y^*\in\gD$ such that $d(x^*, y^*)\ge 4$.
Then (after perhaps interchanging $x$ and $y$)
the map $\gvp_{y^*}\colon N\to \gC_{y^*}$ is a strong valuation-like map.
\end{thm}

The following notation should be compared with \cite[Notation 5.5, p.~944]{RSS}.

\begin{notation}\label{not dotN}
Let $M\subseteq N$ be a subgroup, and let $r, s\in D\sminus N$.
\begin{enumerate}
\item
We denote $\N_M(r)=N(r)\cap M$.
When $M$ is clear from the context we will omit the subscript $M$ and write $\N(r)$
in place of $\N_M(r)$ .
Note that while $N(r)$ is always non-empty, $\N_M(r)$ may well be empty.
If $\N_M(r)=\emptyset$ and $m\in M$, then our convention is that $m\N_M(r)=\emptyset$.

\item
We denote by $\In(r^*,s^*)$ the following relation:
for any $a\in Nr$ and $b\in Ns$ we have $\N_M(a)\subseteq \N_M(b)$ or $\N_M(b)\subseteq \N_M(a)$.

\item
We denote by $\Inc(s^*,r^*)$ the following relation:
$\In(r^*,s^*)$ and for any $b\in \pp_{s^*}$ there exists $a\in \pp_{r^*}$
such that $\N_M(b)\supseteq \N_M(a)$.
\end{enumerate}
Of course, when $M=N$ we have $N(r)=\N_N(r)$.
In this case we abbreviate
\[
In(r^*, s^*)=\mathrm{In}_N(r^*, s^*), \quad Inc(s^*, r^*)=\mathrm{Inc}_N(s^*,r^*)
\]



\end{notation}

Note that in Notation \ref{not dotN}(2) we allow $r^*=s^*$.
Furthermore we have (compare with \cite[Proposition 5.7(2)]{RSS}):

\begin{lemma}\label{lem In(y,y)}
Let $M\subseteq N$ be a subgroup and let $s\in\Dt\sminus N$. Assume that $\In(s^*,s^*)$ holds.
Then the subgroup $\gvp_{s^*}(M)\subseteq\gC_{s^*}$ is totally ordered.
In particular, if $M=N,$ then $(\Gamma_{s^*},\leq_{s^*})$ is a totally ordered group.
\end{lemma}
\begin{proof}
%
Let $m,m'\in M$, and suppose that $\varphi_{s^*}(m)\leq\varphi_{s^*}(m')$ does not hold.
By the equivalence (2)$\Leftrightarrow$(4) in Lemma \ref{lem eq cond}, there exists $b\in\pp_{s^*}$
such that $m\not\in N(m'b)$.
Let $c\in\pp_{s^*}$.
Then $m\in \N_M(mc)\sminus\N_M(m'b)$.
As $mc,m'b\in Ns$, our assumption $\In(s^*,s^*)$ implies that $\N_M(m'b)\subseteq\N_M(mc)$.
Therefore $m'\in \N_M(m'b)\subseteq\N_M(mc)\subseteq N(mc)$.
Using again the equivalence (2)$\Leftrightarrow$(4) in Lemma \ref{lem eq cond}, we conclude that
$\varphi_{s^*}(m')\leq\varphi_{s^*}(m)$.
\end{proof}

Next we have

\begin{prop}[\cite{RS}, Proposition 6.11]\label{prop In(x*,y*)}
Let $x,y \in \Dt\sminus N$ with $d(x^*,y^*)\ge 4$.
Then $In(x^*,y^*)$.
\end{prop}
\begin{proof}
Let $a \in Nx$ and $b\in Ny$.
We need to show that $N(a)\subseteq N(b)$ or $N(b)\subseteq N(a)$.
So assume $N(b)\nsubseteq N(a)$, and let $n\in N(b)\sminus N(a)$.
Let $c:= n^{-1}a$ and $d:= n^{-1}b$.
Then $1\in N(d)\sminus  N(c)$.
By Lemma \ref{lem N(ab)}(3),(1), respectively, $N(c)\subseteq N(cd)\subseteq N(d)$.
Hence also $N(a)\subseteq N(b)$.
\end{proof}

The next two lemmas list some properties of the relations $\In$ and $\Inc$.
They follow \cite[3.10]{Seg2} and \cite[Lemma 6.12]{RS}.

\begin{lemma}\label{lem dotin dotinc}
Let $r,s,t\in \Dt\sminus N$ and let $M\subseteq N$ be a subgroup.
Then,
\begin{enumerate}
\item
if\ $\In(r^*, s^*)$, then
$\Inc(r^*, s^*)$ or $\Inc(s^*, r^*)$;

\item
$\Inc(s^*,r^*)$ and $\In(r^*,t^*)$ imply $\In(s^*,t^*)$;

\item
$\Inc(s^*,r^*)$ and $\Inc(r^*,t^*)$ imply $\Inc(s^*,t^*)$;

\item
if\  $\Inc(s^*, r^*)$, then $\In(s^*, s^*)$.
\end{enumerate}
\end{lemma}
\begin{proof}
(1):\quad
Suppose $\Inc(r^* , s^*)$ does not
hold. Then there is $a \in \pp_{r^*}$ such that $\N(a) \subseteq
\N(b),$ for all $b \in \pp_{s^*},$ so $\Inc(s^* , r^*)$ holds.
\medskip

\noindent
(2):\quad
Let $b\in Ns$ and $c\in Nt$.
Suppose that $\N(b)\nsubseteq \N(c)$.
We need to show that $\N(b)\supseteq \N(c)$.
Let $m \in \N(b)\sminus \N(c)$ .
Then replacing $b$ by $m^{-1}b$ and $c$ by $m^{-1}c$, we may assume that $1\in \N(b)\sminus \N(c)$
(we note that $m \in M,$ so $\N(m^{-1}d) = m^{-1}\N(d),$ for any $d \in \Dt \sminus N$).
By $\Inc(s^*,r^*)$, we can pick $a\in \pp_{r^*}$, with $\N(b)\supseteq \N(a)$.
As $1\in \N(a)\sminus \N(c)$, $\In(r^*,t^*)$ implies that $\N(a)\supseteq \N(c)$, so
$\N(b)\supseteq \N(a)\supseteq \N(c)$, as asserted.

\medskip

\noindent
(3):\quad
First, by (2) we have $\In(s^*,t^*)$.
Let $b\in\pp_{s^*}$.
By $\Inc(s^*,r^*)$, there is $a\in\pp_{r^*}$, with $\N(b)\supseteq \N(a)$.
By $\Inc(r^*,t^*)$, there is $c\in\pp_{t^*}$, with $\N(a)\supseteq \N(c)$.
Thus $\N(b)\supseteq \N(a)\supseteq \N(c)$, and we get $\Inc(s^*,t^*)$.
\medskip

\noindent
(4):\quad
Since $\Inc(s^* , r^*)$ we have $\In(r^*,s^*)$.
Hence (2) (with $t=s$) gives $\In(s^*,s^*)$.
\end{proof}

\begin{lemma}\label{lem consequences of Inc}
Let $r,s\in \Dt\sminus N$ and suppose that $Inc(s^*,r^*)$.
Then
\begin{enumerate}
\item $\le_{s^*}$ is a total order relation;

\item $N_{\le_{s^*}\, 0}\supseteq N_{\le_{r^*}\, 0}$;

\item
there is an epimorphism $\psi\colon\Gamma_{r^*}\to\Gamma_{s^*}$ of partially ordered groups such that
$\gvp_{s^*}=\psi\circ\gvp_{r^*}$.
\end{enumerate}
\end{lemma}
\begin{proof}
(1):\quad
As $Inc(s^*,r^*),$ Lemma
\ref{lem dotin dotinc}(4) gives $In(s^*,s^*)$.
Now use Lemma \ref{lem In(y,y)}.
\medskip

\noindent
(2):\quad
This follows from Corollary \ref{cor properties lf le y}(1).
Indeed, let $b\in \pp_{s^*}$,
and (using $Inc(s^*,r^*)$) pick $a\in\pp_{r^*}$ with $N(b)\supseteq N(a)$.
Then $N(b)\supseteq N(a)\supseteq N_{\le_{r^*}\, 0}$.
As this holds for all $b\in\pp_{s^*}$,
$N_{\le_{s^*}\, 0}\supseteq N_{\le_{r^*}\, 0}$.
\medskip

\noindent
(3):\quad
Recall that for $z=r, s$, the kernel of the map $\gvp_{z^*}\colon N\to\gC_{z^*}$ is $U_{z^*}$.
By (2), $U_{r^*}\subseteq U_{s^*}$, and hence there exists an epimorphism $\psi\colon \gC_{r^*}\to\gC_{s^*},$
such that $\gvp_{s^*}=\psi\circ\gvp_{r^*}$, as required.
Moreover, by (2) again, $\psi$ is a homomorphism of partially ordered groups.
\end{proof}
\medskip

\noindent
\begin{proof}[Proof of Theorem \ref{thm existence of VL in diam 4}]
Let $x,y\in \Dt\sminus N$ such that $d(y^*,x^*)\ge 4$.
By Theorem \ref{thm existence of s-level in diam 3}, both $\varphi_{x^*}\colon N\to\Gamma_{x^*}$ and
$\varphi_{y^*}\colon N\to\Gamma_{y^*}$ are strongly leveled maps.
Thus it remains to show that one of the groups $\gC_{x^*},\gC_{y^*}$ is totally ordered.

Now by Proposition \ref{prop In(x*,y*)}, $In(x^*,y^*)$.
By Lemma \ref{lem dotin dotinc}(1),
after perhaps interchanging $x$ and $y$, we may assume that $Inc(y^*,x^*)$.
Then by Lemma \ref{lem consequences of Inc}(1), $\gC_{y^*}$ is a totally ordered group.
\end{proof}

\section{The case where ${\rm diam}(\gD)\ge 5$}
In this section we continue the notation and hypotheses of \S4.
The purpose of this section is to prove Theorem B(3) of the Introduction.
We thus assume that $x^*, y^*\in\gD$ are such that $d(x^*,y^*)\ge 5$.
By Lemma \ref{lem dotin dotinc}(1), we may assume without loss of generality
that $Inc(y^*, x^*)$ holds.
We will prove

\begin{thm}\label{thm level 0 in diam ge 5}
The map $\gvp_{y^*}\colon N\to\gC_{y^*}$ is a strong valuation like
map of s-level $0.$
\end{thm}
%
%
By Lemma \ref{lem dotin dotinc}(4), $In(y^*,y^*)$ holds,
and this together with Lemma \ref{lem basic properties of N(y)}(1) and the fact that $-1\in N$ implies that
%
%
\begin{equation}\label{eq N(g)=N(-g)}
N(g)=N(-g),\qquad\text{for all $g\in Ny$}.
\end{equation}
%
%
\begin{lemma}\label{lem N(a)subseteq N(b)}
Let $a \in xN,$ $b \in yN$ such that $N(a)\subseteq N(b)$ and let $n\in N\sminus N(b)$.
Then
\[
N(a)\subseteq N(a + n) \cap N(b + n).
\]
\end{lemma}
\begin{proof}
Since $N(b)=N(-b)$ we have $N(b)=-N(b)$ and $N(-a)\subseteq N(b)$.

If $a-n\in N$, then $-n\in N(a)\subseteq N(b)=-N(b)$, a contradiction.
Thus $a-n\not\in N$, so by Remark \ref{rem a-b in N}(1), $d(a^*,(a-n)^*)\leq1$.
Also, $b+n\not\in N$ and we similarly have $d((b+n)^*,b^*)\leq1$.
As $d(a^*,b^*)\geq5$ we conclude that $d((a-n)^*,(b+n)^*)\geq3$.
Hence, by Lemma \ref{lem N(x+y)}(2), $N(a+b)=N(a-n)\cap N(b+n)$.
In addition, Lemma \ref{lem N(x+y)}(2) implies that $N(a+b)=N(a)\cap N(b)=N(a)$.
Therefore
\begin{equation}\label{eq N(a)=N(a-n) cap N(b+n)}
N(a)=N(a-n)\cap N(b+n).
\end{equation}
The same argument, with $-a$ in place of $a$, shows that
%
%
\begin{equation}
\label{eq N(-a)=N(-a-n) cap N(b+n)}
N(-a)=N(-a-n)\cap N(b+n).
\end{equation}
%
%
Altogether equations (\ref{eq N(a)=N(a-n) cap N(b+n)}) and (\ref{eq N(-a)=N(-a-n) cap N(b+n)}) imply the assertion.
\end{proof}
\medskip

\begin{proof}[Proof of Theorem \ref{thm level 0 in diam ge 5}]
We denote $\le\,=\,\le_{y^*}$ and $\gvp=\gvp_{y^*}$.
By Theorem \ref{thm existence of VL in diam 4}
it is enough to show that $\gvp$ is a strongly leveled map having s-level $0,$ i.e.
%
%
\begin{equation}\label{eq level 0}
1 + N_{> 0} \subseteq N_{\le 0}
\end{equation}
%
%
(we notice that since $\gC_{y^*}$ is totally ordered, we have $\pm N_{>0} = N_{> 0},$
hence equation (SL) in subsection \ref{sub vlm} of the introduction
simplifies to equation \eqref{eq level 0}).

We first show that
\begin{equation}\label{eq 1 + N > 0 subseteq N}
1 + N_{> 0} \subseteq N.
\end{equation}
To this end let $n \in N_{> 0}$.
It follows from Corollary \ref{cor properties lf le y}(1) that there exists $b\in \pp_{y^*}$
such that $n \notin N(b)$.
Using $Inc(y^*,x^*)$ we pick $a \in \pp_{x^*}$ such that $N(b)\supseteq N(a)$.
Since $1 \in N(a)$,  Lemma \ref{lem N(a)subseteq N(b)} shows that both $a + n + 1$ and $b + n + 1$ belong to $N$.
If $n + 1 \notin N$, then $a^*$, $(n + 1)^*$, $b^*$ would be a path in $\gD$, contrary to $d(a^* , b^*) \ge 5$.
Thus $n + 1 \in N$ proving  equation \eqref{eq 1 + N > 0 subseteq N}.

Next we note that $1-(1+a)=-a\not\in N$.
By equation \eqref{eq 1 + N > 0 subseteq N},   $-(1 + a)\in N_{\le 0}$,  whence also $1+a\in N_{\le 0}$
 (see equation \eqref{eq N(g)=N(-g)}).

Corollary \ref{cor properties lf le y}(1) now implies that $1+a\in N(z)$ for every $z\in\pp_{y^*}$.
Equivalently, $z+1\in N(a)$ for every such $z$.
By Lemma \ref{lem N(a)subseteq N(b)},
\[
a + (z + 1 + n) \ , \ b + (z + 1 + n) \ \in \ N.
\]
The assumption $z + 1 + n \notin N$ would lead us again to the false
conclusion that $d(a^* , b^*) \le 2$.
Thus, $z + (1+n) \in N$.
Since this is true for all $z \in \pp_{y^*}$, we see that $1+n \in N_{\le 0}$, by Corollary \ref{cor properties lf le y}(1), proving  equation \eqref{eq level 0}.
\end{proof}

\section{Property \ETH}\label{section 3.5}
In this section we continue the notation and hypotheses
of \S 4.
Recall that we denote by $F$ the center of $D$.
Throughout this section $\gS\le {\rm Sym}(\Dt)$ is a permutation group on $\Dt$.
For $x\in\Dt$ and $\gs\in\gS,$ let $\gs(xN):=\{\gs(xn)\mid n\in N\}$.
We assume that $\gs(N)=N,$ for all $\gs\in\gS,$ and
\[
\gs(xN)=\gs(x)N,\quad\text{for all }x\in\Dt\text{ and }\gs\in\gS,
\]
that is $\gs(x^*)=\gs(x)^*,$ for all $x\in\Dt$ and $\gs\in\gS$.
We use the letter $\gS$ to also denote the group
of permutations of $\gD$ induced by $\gS$.
%

\begin{remark}
Note that if $\Sigma \le\Aut(\Dt)$ is a subgroup that normalizes
$N,$ then of course $\Sigma$ satisfies our hypothesis.
\end{remark}

\begin{Def}\label{def eth}
We say that $\gD$ satisfies {\it Property \ETH\ with respect to $x^*, y^*\in\gD,$
the subgroup $\gS\le {\rm Sym}(\Dt),$ and the subgroup $M\subseteq N,$} if $-1\in M$ and
\begin{enumerate}
%
\item
$\gs(a+k)^*=(\gs(a)+k)^*$, for every $\gs\in\gS,$ $a\in \Dt\sminus N$ and $k\in M$.

\item
$d(x^*, y^*)\ge 3$.

\item
For {\it every} $r^*, s^*\in\gD$ such that $x^*, r^*, s^*, y^*$
is a path in $\gD,$ there exists $\gs\in\gS$ such that  $d(\gs(x^*),y^*)\ge 3,$ and
$\gs(x)^*, \gs(r)^*, s^*, y^*,$ is {\it not} a path in $\gD$.
\end{enumerate}
\end{Def}

The purpose of this section is to prove:

\begin{thm}\label{thm 3.5}
Assume that
$\gD$ satisfies property \ETH\ with respect to $x^*, y^*\in\gD,$
the subgroup $\gS\le{\rm Sym}(\Dt),$ and the subgroup  $M\subseteq N$.
Then, after perhaps interchanging $x^*$ and $y^*,$ the map
$\gvp_{y^*}\colon N\to\gC_{y^*}$ is a strongly leveled map
such that the subgroup $\gvp_{y^*}(M)$ of\, $\gC_{y^*}$ is totally ordered.
\end{thm}

 \begin{lemma}\label{lem not eth}
 Suppose that $\Delta$ satisfies property \ETH \ with respect to $x^*,y^*\in\Delta,$ $\gS$ and $M$.
 Then $\In(x^*,y^*)$ holds.
 \end{lemma}
 \begin{proof}
 Suppose that $\In(x^*,y^*)$ does not hold.
 Then there exist $a\in xN$ and $b\in yN$ such that $\N(a)\not\subseteq\N(b)$ and $\N(b)\not\subseteq\N(a)$.
 Pick $m\in\N(a)\sminus\N(b)$.
 We may replace $a, b$ by $-m^{-1}a, -m^{-1}b$ respectively to assume that
 \[
 -1\in \N(a)\sminus \N(b).
 \]
 Also pick
 \[
 k \in \N(b)\sminus \N(a),
 \]
 and set $r=a+k$ and $s=b-1$.
 Then $r,s\not\in N$.

 Take $\gs\in\Sigma$ such that $d(\gs(x)^*,y^*)\ge 3$.
 We show that
 \[
 \tag{$\calp$}
 \gs(x)^*, \gs(r)^*, s^*, y^*
 \]
 is a path  in $\Delta$.
 In particular, applying this for $\gs={\rm id}$, we obtain from (2) of Definition \ref{def eth} that $x^*,r^*,s^*,y^*$ is a path in $\Delta$, which is a contradiction to (3) of Definition \ref{def eth}.

 Now since $d(\gs(x)^*,y^*)\ge 3$, it is enough to show that in each step in $\mathcal{P}$ the distance
 is {\it at most} $1$.

 Indeed, by (1) of Definition \ref{def eth}, $\gs(r)^*=\gs(a+k)^*=(\gs(a)+k)^*$.
 Since $k\in M\subseteq N$, (V1) implies that
 \[
 d(\gs(x)^* , \gs(r)^*)=d(\gs(a)^* , (\gs(a)+k)^*)\le 1.
 \]

 Next, we have  $a-1\in N$, which by  (1) of Definition \ref{def eth} implies that
 $(\gs(a)-1)^*=\gs(a-1)^*=1^*$.
 Therefore $-1\in N(\gs(a))$.
 Since $d(\gs(a)^*,b^*)=d(\gs(x)^*,y^*)\ge 3$, we deduce from Lemma 4.4(3)  that
 $k\in N(b)\subseteq N(b\gs(a))$, and hence $b\gs(a)+k\in N$.
 Therefore, by (V1),
 \[
 d(b^*(\gs(a) + k)^* , (b - 1)^*)=\\
 d((b\gs(a) + k) + (b - 1)k)^* , ((b - 1)k)^*)\le 1.
 \]
 In addition $d(b^*, (b-1)^*) \le 1$, so by (V3),
 \[
 d(\gs(r)^*,s^*)=d((\gs(a) + k)^* , (b - 1)^*)\le 1.
 \]

 Finally, (V1) implies that
 \[
 d(s^* , y^*)=d((b - 1)^* , b^*)\le1,
 \]
 as required.
 \end{proof}

\medskip

\noindent
\begin{proof}[Proof of Theorem \ref{thm 3.5}]
Since $d(x^*,y^*)\geq3$, Theorem \ref{thm existence of s-level in diam 3} implies that both $\gvp_{y^*}$
and $\gvp_{x^*}$ are strongly leveled maps.
By Lemma  \ref{lem not eth}, $\In(x^*, y^*)$ holds.
In view of Lemma \ref{lem dotin dotinc}(1), we may assume (after perhaps
interchanging $x^*$ and $y^*$) that $\Inc(y^*, x^*)$ holds.
Then Lemma \ref{lem dotin dotinc}(4) shows that  $\In(y^*, y^*)$ also holds.
By Lemma \ref{lem In(y,y)}, the subgroup  $\gvp_{y^*}(M)$ of $\gC_{y^*}$ is totally ordered.
\end{proof}

Using Theorem \ref{thm 3.5} we will prove at the end of \S10 (compare with \cite[Theorem 1, p.~931]{RSS}):

\begin{thm}\label{thm v-3.5}
Let $D$ be a finite-dimensional separable (but not
necessarily central)
division algebra over an infinite field $K$ of finite transcendence degree over its prime field, and let $N \subseteq D^{\times}$ be a normal
subgroup of finite index containing $-1$. Assume that
$D^{\times}/N$ supports a V-graph $\gD$ that satisfies property \ETH\
with respect to $x^*, y^*\in\gD,$ the subgroup $\gS\le{\rm Sym}(\Dt),$ and the subgroup $M=N\cap k^{\times}$.

Then there exists a non-empty finite set $\widetilde{T}$ of non-trivial valuations
of $D$ such that $N$ is open in $D^{\times}$ with respect to the $\widetilde{T}$-adic topology.
\end{thm}

\section{Valuations on division algebras}
\label{vals on division algebras}
In this section we recall some notion and facts on valuations
on division algebras, which will be needed in the next two sections.

We recall that a valuation on a division ring $D$ is a {\it surjective} group
homomorphism $v \colon D^{\times} \to \Gamma_v$ onto a totally
ordered group $\Gamma_v$ (the {\it value group}) such that
\[
v(x + y) \ge \min\{v(x) , v(y)\} \ \ \text{whenever} \ \ x + y \neq 0.
\]
We let
\[
\mathcal{O}_{D , v} = \{ x \in D^{\times} \: \vert \: v(x) \ge 0\} \cup \{ 0 \}
\]
denote the corresponding valuation ring, and $\mathfrak{m}_{D , v}$ its maximal ideal.
More generally, for $\delta\in (\Gamma_v)_{\geq 0}$, we define the following two-sided ideal of
$\mathcal{O}_{D , v}$:
\[
\mathfrak{m}_{D , v}(\delta) = \{ x \in D^{\times} \: \vert \: v(x)
> \delta \} \: \cup \: \{ 0 \}
\]
(so that $\mathfrak{m}_{D , v} = \mathfrak{m}_{D , v}(0)$).


Recall that the {\it $v$-adic topology} on $D$ is the ring topology which has the ideals
$\mathfrak{m}_{D , v}(\delta)$ for $\delta \in (\Gamma_v)_{\geq0}$ as
 a fundamental system of neighborhoods of zero (see \S 5, n$^{\circ}$1 in \cite[Ch.\ 6]{Bo}).
This topology turns $D^{\times}$ into a topological group, and the openness of a subgroup $N \subseteq
D^{\times}$ in the $v$-adic topology is equivalent to the existence of $\delta \in
(\Gamma_v)_{\geq 0}$ such that $1 + \mathfrak{m}_{D , v}(\delta) \subseteq N$.

More generally, given a finite set $T = \{ v_1, \ldots , v_r \}$ of valuations of $D$,
and $\delta_i \in (\Gamma_{v_i})_{\ge 0}$, $i=1, \ldots , r$, we define
\[
\mathfrak{m}_{D,T}(\delta_1, \ldots , \delta_r) = \bigcap_{i = 1}^r
\mathfrak{m}_{D,v_i}(\delta_i).
\]
Clearly, $\mathfrak{m}_{D,T}(\delta_1, \ldots , \delta_r)$ is a two-sided
ideal of $\mathcal{O}_{D , T} = \bigcap_{v \in T} \mathcal{O}_{D,v}$.
The ideals $\mathfrak{m}_{D,T}(\delta_1, \ldots , \delta_r)$ form a
fundamental system of neighborhoods of zero for a topology on $D$
compatible with the ring structure;
this topology will be called $T$-{\it adic}.
Thus, a subgroup $N \subseteq D^{\times}$ (resp.~a subring $R\subseteq D$)
is $T$-adically open iff it contains the congruence subgroup $1 +
\mathfrak{m}_{D,T}(\delta_1, \ldots , \delta_r)$ (resp.~the ideal $\mathfrak{m}_{D,T}(\delta_1, \dots , \delta_r)$)
for some $\delta_i \in (\Gamma_{v_i})_{\ge 0}$, $i = 1, \ldots , r$.

Next let $\Gamma=(\Gamma,\leq)$ be a totally ordered commutative group.
We recall that the {\it height} (also called {\it rank}) of $\Gamma$ is  the supremum of all non-negative integers $r$
such that there exist epimorphisms
\[
\Gamma=\Gamma_0 {\smash{\mathop{\longrightarrow}\limits^{\mu_1}}}
\Gamma_1{\smash{\mathop{\longrightarrow}\limits^{\mu_2}}} \cdots\rightarrow^{}\Gamma_{r-1}
{\smash{\mathop{\longrightarrow}\limits^{\mu_r}}}\Gamma_r=\{0\}
\]
of totally ordered groups, where $\mu_1,\ldots,\mu_r$ have non-trivial kernels.
The group $\Gamma$ has height $\leq1$ if and only if it embeds in the ordered additive group of $\mathbb{R}$
(\cite{Efr_book}, Th.\ 2.5.2).
We note that if $\Gamma$ is commutative and has finite height, then there is  an epimorphism $\mu\colon \Gamma\to\bar\Gamma$,
with $\bar\Gamma$ of height one.

The {\it height} of a valuation  $v$ is defined to be the height of its value group $\Gamma_v$
(cf. \cite{Bo}, ch.~6, \S4, n$^{\circ}$4 or \cite{Efr_book}, \S2.2, for a discussion on the height of a totally
ordered group/valuation).

The following definition describes a useful connection between leveled maps and valuations.

\begin{Def}
\label{def assoc val}
Let $N$ be a subgroup of $D^\times$, $\Gamma$ a partially ordered group,  $\varphi \colon N \to
\Gamma$ a group homomorphism, and  $v \colon D^{\times} \to \Gamma_v$ a valuation.
We say that $v$ is \emph{associated} with $\varphi$ if $\varphi(n)\ge 0$ implies $v(n)\ge 0$, for all $n\in N$.
\end{Def}

\begin{remarks}[\cite{RS}, Remarks 2.5]\label{rem assoc val}
\begin{enumerate}
\item
Given a non-trivial homomorphism $\varphi \colon N \to \gC$, the non-trivial valuation
$v\colon \Dt \ra \gC_v$ is associated with $\varphi$ if and only if
there exists a non-trivial homomorphism $\theta \colon \varphi(N) \to
\Gamma_v$ of ordered groups such that the square
\[
\xymatrix{
N\ar[r]^{\varphi}\ar@{^{(}->}[d]_{\iota} & \varphi(N)\ar[d]^{\theta}\\
D^{\times}\ar[r]^v &\Gamma_v\\
}
\]
in which $\iota$ is the inclusion map, commutes.
In fact, this was the original definition used in \cite{RS}, \cite{RSS}, \cite{R}.
\item
If $v\colon D^\times\to \Gamma_v$ is a valuation and $\mu \colon \Gamma_v \to \bar\Gamma$
is a surjective homomorphism of totally ordered groups,  then $\bar v = \mu \circ v\colon D^\times\to \bar\Gamma$
is also a valuation.
Also, if $v$ is associated with $\varphi\colon N\to\Gamma$ with respect to $\theta$, then
$\bar v$ is also associated with $\varphi$, with respect to $\mu\circ\theta$.
Further, every $v$-adically open subgroup of $D^\times$ is also $\bar v$-adically open.
\end{enumerate}
\end{remarks}

Given valuations $u,u'$ on $D$ with value groups
$\Gamma_u,\Gamma_{u'}$, respectively,
we say that $u'$ is {\it coarser} than $u$  if there is an epimorphism of
totally ordered groups $\mu\colon\Gamma_u\to\Gamma_{u'}$
such that $u'=u\circ\mu$.

\begin{lemma} \label{lemma on coarsenings}
Let $D$ be a division ring which is finite-dimensional over its center.
Let $r\geq2$ and let $w_1,\ldots, w_r$ be distinct valuations of height $1$ on $D$.
Let $N_1,\ldots,N_r$ be proper subgroups of $D^\times$ which are open in the
$w_i$-adic topologies, $i=1,2,\ldots,r$, respectively.
Let $N=N_1\cap\cdots\cap N_r$.
Then $N$ is not $u$-adically open for any nontrivial valuation $u$ on $D$.
\end{lemma}
\begin{proof}
Assume the contrary.
The assumption on $D$ implies that $\Gamma_u$ is commutative (see \cite{JacobWadsworth}, p.\ 628).
The set of valuations on $D$ which are coarser than $u$ is linearly ordered
with respect to the coarsening relation (see e.g., \cite{Efr_book}, Prop.\ 2.1.3 and Prop.\ 2.2.1).
Since $w_1,\ldots,w_r$ have height $1$ and are distinct, none of them is coarser than the other.
Therefore at most one of them can be coarser than $u$.
Without loss of generality $w_1$ is not coarser than $u$.
Since it has rank $1$, there is no nontrivial common coarsening of $w_1$ and $u$,
i.e., they are independent valuations.

As $N_1$ contains $N$, it is open in both the $w_1$-adic topology and the $u$-adic topology.
Furthermore, for every $d\in D^\times$ the coset $dN_1$ is $u$-adically open in $D^\times$.
By the weak approximation theorem for independent valuations \cite{Marshall},
$N_1\cap dN_1\neq\emptyset$, so $d\in N_1$.
This contradicts the assumption that $N_1\neq D^\times$.
\end{proof}

\section{Valuations from strongly leveled maps}\label{sect val from s leveled maps}

The goal of this section is to show that  strongly leveled maps, under some additional assumptions, give rise to valuations.
The main result of this section is  the following theorem.
We recall from subsection \ref{sub vlm} of the Introduction the notion of a {\it leveled}, {\it strongly leveled}, and {\it strong valuation-like} map.

\begin{thm}\label{thm val}
Let $D$ be a finite-dimensional  (but not necessarily central)
 division algebra over a field $k$ of finite transcendence degree over its prime field, and let $N \subseteq D^{\times}$ be a normal subgroup containing $-1$ of finite index.
Let $\varphi \colon N \to \Gamma$ be a strongly leveled map, where $(\Gamma,\le)$ is a partially ordered group,
such that the subgroup $\varphi(N\cap k^{\times})$ of $\Gamma$ is totally ordered.
Then:
\begin{enumerate}
\item[(a)]
the restriction $\varphi_k =\varphi \restr_{N \cap k^{\times}}$ is a strong valuation-like map;
\item[(b)]
there exists a height one valuation $v$ of $k$ associated with
$\varphi_k$ such that $N\cap k^{\times}$ is open in $k^\times$ in the $v$-adic topology.
\end{enumerate}
\end{thm}

Let us first make the connection between
a strongly leveled and a leveled map.
\begin{lemma}\label{lem s-level and level}
Let $D$ be an infinite division ring, let $N\subseteq \Dt$ be a finite index normal subgroup containing $-1,$ and
let \mbox{$\gvp\colon N\to\gC$} be a strongly leveled map
of s-level $\ga\in \gC_{\ge 0}$.  Then
\begin{enumerate}
\item
if $\gb\in\gC_{\ge\ga}$ is such that $N_{> \gb}\ne\emptyset,$
then $\gb$ is an s-level of $\gvp$;

\item
$\gvp$ is a leveled map of level $\ga$.
\end{enumerate}
\end{lemma}
\begin{proof}
Part (1) follows immediately from the definitions.
For part (2), note that $N_{< -\ga}\ne\emptyset$ since $N_{> \ga}\ne\emptyset$.
Let $n \in N_{< -\ga}$.
Then $1 + n^{-1} \in 1 + N_{> \ga} \subseteq N_{\le 0}$, and therefore
\[
\gvp(1 + n) = \gvp(n(1 + n^{-1})) \le \varphi(n) < -\ga,
\]
i.e.,
$1 + n \in N_{< -\ga}$.
Thus, $\gvp$ is a leveled map of level $\ga$.
\end{proof}

We record the following two results from \cite{R} and \cite{RSS}, respectively.
Here $D$ is a division ring and $N$ a normal subgroup of $D^\times$ containing $-1$.

\begin{prop}[\cite{R}, Proposition 3]
\label{prop open}
Let $\varphi \colon N \to \Gamma$ be a leveled map, and $T$ a finite set of
valuations of $D$ associated with $\varphi$.
Assume that there is a $T$-adically open subring $\mathcal{M}$ of $D$ such that
\[
\mathcal{M} \cap N \subseteq N_{> - \beta}
\]
for some $\beta \in\Gamma_{> 0}$.
Then $N$ is $T$-adically open in $D^{\times}$.
\end{prop}
\medskip

\begin{lemma}[\cite{RSS}, Proposition 4.2]
\label{RSS prop 4.2}
Let $\varphi\colon N\to\Gamma$ be a strongly leveled map with s-level $\alpha$.
Let $\mathcal{A}$ (resp., $\mathcal{R}$) be the subring of $D$ generated by $N_{>\alpha}$ (resp., $N_{\geq0}$).
Then $-1\not\in\mathcal{A}$, and for every $m\in \mathcal{A}\cap N_{\geq0}$, the element $\varphi(m)$ is an s-level of $\varphi$.
\end{lemma}

Note that $\mathcal{R}$ (resp., $\mathcal{A}$) coincides with the set of all elements of the form $\epsilon_1a_1 + \cdots +  \epsilon_l a_l$ with $\epsilon_i = \pm 1$ and $a_i \in N_{\ge 0}$ (resp., $a_i \in N_{> \alpha}$), and that $\mathcal{A}$ is in fact a ring {\it without identity}.

The following result establishes the existence of a valuation
associated to a given strongly leveled map in the simplest case where the
division algebra is assumed to be commutative and the map $\varphi$ to be strong valuation-like.

\begin{thm}
\label{thm val comm}
{\rm (commutative case; see \cite{RS}, Theorem 4.1)}
Let $K$ be a field, and let $N \subseteq K^{\times}$ be a subgroup of finite index containing $-1$.
Assume that $\varphi \colon N \to \Gamma$ is a strong valuation-like map.
Then
\begin{enumerate}
\item
there exists  a non-trivial valuation $v$ on $K$ associated with $\varphi$ such that $N$ is open in the $v$-adic topology;

\item
additionally, if $K$ has finite transcendence degree over its prime field,
then $v$ can be taken to have height one;

\item
the subring $\mathcal{R}$ of $K$ generated by $N_{\ge 0}$ is $v$-adically open.
\end{enumerate}
\end{thm}
\begin{proof}
Let $\ga$ be an s-level of $\gvp$, and let $\mathcal{R}$ and $\mathcal{A}$ be the subrings introduced in Lemma \ref{RSS prop 4.2} for $K$ in place of $D$.
Let $\widetilde{R}$ be the integral closure of $\mathcal{R}$ in $K$.

We claim that $\widetilde{\mathcal{R}}$ is a valuation ring.
Indeed, given any $x \in K^{\times}$,
for $m = [K^{\times} : N]$, we have $x^m \in N$, hence either $x^m$
or $x^{-m}$ is in $N_{\ge 0} \subseteq \mathcal{R}$. Then,
respectively, either $x$ or $x^{-1}$ is in $\widetilde{\mathcal{R}}$.

Next, we show that
\begin{equation}\label{eq RcapN<ga}
\widetilde{\mathcal{R}} \cap N_{< -\alpha} = \emptyset.
\end{equation}
Since $N_{<-\alpha}\neq\emptyset$, this will also show that
the valuation ring $\widetilde{\mathcal{R}}$ is a proper subring of $K$.
Suppose that $z \in \widetilde{\mathcal{R}}\cap N_{< -\alpha}$.
Then $z$ satisfies a polynomial equation
\[
z^d + a_1z^{d-1} + \cdots + a_d = 0
\]
with $a_i\in\mathcal{R}$.
Since $z^{-1}\in N_{> \alpha},$ it follows that
\[
-z = a_1 + \cdots + a_d z^{-(d-1)} \in \mathcal{R}.
\]
On the other hand, the inclusion $N_{\ge 0} N_{> \alpha} \subseteq
N_{> \alpha}$ implies that $\mathcal{R}\mathcal{A} \subseteq
\mathcal{A}$, i.e. $\mathcal{A}$ is an ideal of $\mathcal{R}$.
Since $z^{-1} \in N_{> \alpha} \subseteq \mathcal{A}$, we obtain $-1 =
-z \cdot z^{-1} \in \mathcal{A}$, contrary to Lemma \ref{RSS prop 4.2}.

Let $v$ be the valuation of $K$ corresponding to the valuation ring $\widetilde{\mathcal{R}}$,
 i.e., $x\in \widetilde{\mathcal R}$ if and only if $v(x)\geq0$, for $x\in K^\times$ (cf.\ \cite{Bo, Efr_book}).
Since $N_{\ge 0} \subseteq \mathcal{R} \subseteq \widetilde{\mathcal R}$,
for any $x \in N_{\ge 0}$ we have $v(x) \ge 0$,
which implies that $v$ is associated with $\varphi$.
By construction, $\widetilde{\mathcal{R}}$ is $v$-adically open in $K$.
Equation \eqref{eq RcapN<ga} implies that for any $\beta > \alpha$ we have
\[
\widetilde{\mathcal R} \cap N \subseteq N_{> -\beta},
\]
so the $v$-adic openness of $N$ in $K^\times$ follows from Proposition \ref{prop open}.
This proves (1).
\medskip

For (2), suppose that $K$ has finite transcendence degree over it prime field.
It follows from Cor.\ 1 in n$^{\circ}$3 and Prop.~3 in n$^{\circ}$2 of \cite{Bo}, Ch.~6, \S10,
that the (commutative) value group $\Gamma_v = v(K^{\times})$ has finite height.
As we have observed in Section \ref{vals on division algebras}, there is an epimorphism $\mu\colon\Gamma_v\to \bar\Gamma$ of totally ordered groups,
with $\bar\Gamma$ of height one.
By Remark \ref{rem assoc val}(2), $\bar v=\mu\circ v\colon K^\times\to \bar\Gamma$ is a valuation of height one which
is associated with $\varphi$, and since $N$ is $v$-adically open, it is also $\bar v$-adically open.

Finally, we prove (3).
By (1), there exists a non-negative $\delta \in \Gamma_v=v(K^{\times})$ such that
$1+ \mathfrak{m}_{K , v}(\delta) \subseteq N$.
Pick any $c \in N$ with $v(c) > 0$.
Then for any
\[
x \in c(1 + \mathfrak{m}_{K , v}(\delta))
\]
we have $v(x) > 0$.
Since $\varphi(N)$ is totally ordered and $v$ is
associated with $\varphi$, this implies that $\varphi(x) > 0$.
Thus
\[
c(1 + \mathfrak{m}_{K , v}(\delta)) \subseteq N_{> 0} \subseteq \mathcal{R}.
\]
Setting $\delta'=v(c)+\delta$ we obtain from this and from $c\in N_{>0}\subseteq \mathcal{R}$ that
\[
\mathfrak{m}_{K , v}(\delta') = c\mathfrak{m}_{K , v}(\delta) \subseteq
\mathcal{R},
\]
proving that the ring $\mathcal{R}$ is $v$-adically open in $K$.
\end{proof}

\begin{lemma}\label{quotients}
Let $D$ be a division ring and let $N$ be a normal subgroup of $D^\times$ of finite index.
Let $\Gamma$ be a totally ordered group and $\varphi\colon N\to\Gamma$ a homomorphism.
Let $\mathcal{R}$ be the subring of $D$ generated by $N_{\ge 0}$.
Then any $x \in D$ can be written in the form
$x = a b^{-1}$ with  $a \in \mathcal{R}$ and $b \in N_{\geq 0}$.
\end{lemma}
\begin{proof}
By \cite[Theorem 1, p.\ 377]{Tu}, we can write $x = n_1 - n_2$ with $n_1 , n_2 \in N$.

Suppose first that $\varphi(n_1) \geq \varphi(n_2)$.
When $\varphi(n_2) \geq 0$, we take $a=x$ and $b=1$.
When $\varphi(n_2) < 0$ we take $a = n_1 n^{-1}_2 - 1$ and $b = n^{-1}_2$.

The case $\varphi(n_1) \leq \varphi(n_2)$ is proved similarly.
\end{proof}

\begin{proof}[Proof of Theorem \ref{thm val}]
In view of Theorem \ref{thm val comm} (with $k,N\cap k^\times$ in place of $K,N$),
it suffices to prove (a), that is,  that $\varphi_k$ is a strongly leveled map.
For this, we let $\alpha\in\Gamma_{\geq0}$ be an s-level for $\varphi$.
Let $\mathcal{A}$ denote the subring of $D$ generated by $N_{> \alpha}$,
and let $\mathcal{R}_0$ be  the subring  of $k$ generated by  $(N\cap k^{\times})_{\geq 0}$.
Pick an arbitrary $s \in N_{> \alpha}$, and let
$p(t) = c_{\ell} x^{\ell} + c_{\ell - 1} x^{\ell - 1} +  \cdots + c_0$ be a minimal polynomial of $s$ over $k$.
By Lemma \ref{quotients}  (for the subgroup $N\cap k^\times$ of $k^\times$),
we can write $c_i=a_ib_i^{-1}$ with $a_i\in\mathcal{R}_0$
and $b_i \in (N \cap k^{\times})_{\geq 0}$, $i=0,1,\ldots,\ell$.
Multiplying by $b_0\cdots b_\ell$, we may therefore assume that $c_i \in \mathcal{R}_0$ for every $i$.
Of course, $c_0 \neq 0$, and we have
\[
c_0 = -(c_{\ell} s^{\ell} + \cdots + c_1 s) \in \mathcal{A}.
\]
Then for $d = [D^{\times} : N]$ we have
\[
b := (c_0)^d \in \mathcal{A} \cap N \cap k^{\times}.
\]

We claim that $b \in N_{> 0}$.
Indeed, since the group $\varphi(N\cap k^{\times})$ is totally ordered,
we would otherwise have $b^{-1} \in N_{\geq 0}$.
As before, the inclusion $N_{\ge 0}N_{>\alpha}\subseteq N_{>\alpha}$ implies that
$\mathcal{R}_0\mathcal{A}\subseteq\mathcal{A}$.
We obtain that $-1 = b^{-1}(-b) \in \mathcal{A}$, which contradicts the first part of Lemma \ref{RSS prop 4.2}.

Thus, $b \in \mathcal{A} \cap N_{> 0}$.
By the second part of Lemma \ref{RSS prop 4.2}, $\beta = \varphi(b)$ is an s-level for $\varphi$.
It is therefore also an s-level for $\varphi_k$.
\end{proof}

\section{Openness with respect to finitely many valuations}
Throughout this section we consider a finite-dimensional division algebra $D$ over a field $k$.
We assume that $D$ is {\it separable} over $k$, i.e., the center $F=Z(D)$ of $D$ is
a separable algebraic extension of $k$. 
We further assume that $k$ is equipped with a height one valuation $v$.
Since $D$ is finite-dimensional over $k$, the field extension $F/k$ is finite,
and there are at most $[F:k]$ extensions of $v$ to $F$  (see \cite{Bo}, Ch.\ 6, \S8, n{$^\circ$}3, Th.\ 1).
Each such extension is also of height one (\cite{Efr_book}, Cor.\ 14.2.3(c)).

Let $\vert \ \vert_v$ denote the absolute value on $k$ associated with $v$.
Given a basis $\omega_1, \ldots , \omega_m$ of $D$ over $k$, we define a norm on $D$ by
\begin{equation}
\label{eq norm}
\vv a_1\omega_1 + \cdots + a_m\omega_m \vv_v := \max_{i = 1, \ldots
, m} \vert a_i \vert_v.
\end{equation}
This norm turns $D$ into a normed vector space over $(k,\vert \ \vert_v)$.
Let $\tau_v$ denote the induced topology on $D$.
One easily verifies that $\tau_v$ and the notion of boundedness on $D$ associated with $\vv \ \vv_v$
do not depend on the choice of the basis.

Next let   $S$  be a non-empty finite set of height one valuations on $F=Z(D)$.
Let $\kappa_1, \ldots , \kappa_{n^2}$  be a fixed basis of $D$ over $F$.
For $w \in S$ we may define a norm $\vv \ \vv_w$ on $D$ with respect to this basis in a way similar to \eqref{eq norm}.
Then
\[
D_w = D \otimes_F F_w,
\]
is a finite-dimensional algebra over the completion $F_w$ of  $F$ with respect to $\vert \ \vert_w$.
The norm $\vv \ \vv_w$ extends to $D_w,$ it is defined exactly as in \eqref{eq norm}
using the basis $\omega_1\otimes 1,\ldots,\omega_m\otimes 1$ of $D_w$ over $F_w$.
We endow $D_w$ with the corresponding topology and the notion of boundedness.
Now set
\[
D_S = \prod_{w \in S} D_w,
\]
and endow it with the product topology $\tau_S$.
We have a diagonal embedding $\iota_S\colon D\to D_S$.
The topology on $D$ induced from $\tau_S$ via $\iota_S$ is then the
{\it $\tau_S$-topology on $D$}.
It restricts to the $S$-adic topology on $F$.
Let $\mathrm{pr}_w\colon D_S\to D_w$, for $w\in S$, and $\mathrm{pr}_T\colon D_S\to D_T$,
for $T\subseteq S$, be the projection maps (with the usual convention that $D_T$ is a singleton if $T$ is empty).

We will need the following generalization of \cite{R}, Lemma 2:

\begin{lemma}
\label{lem pr}
In the above setup, let $\mathcal{B}$ be a $\tau_S$-open subring of $D_S$.
Let $S_0$ be a subset of $S$ such that  $\mathrm{pr}_{w}(\mathcal{B})$ is unbounded for every $w\in S_0$,
and set $T:=S\sminus S_0$.
Then
\[
\mathcal{B} = \mathrm{pr}_T(\mathcal{B}) \times D_{S_0}.
\]
In particular, if $T = \emptyset$ then $\mathcal{B} = D_{S} = D_{S_0}$.
\end{lemma}
\begin{proof}
When $|S_0|=1$ this is proved (in an equivalent form) in \cite[Lemma 2]{R}.
The general case follows by induction.
\end{proof}
\medskip

We will also need the following fact from \cite{RSS}:

\begin{prop}[\cite{RSS}, Theorem 2.4]
\label{thm RSS}
Let $D$ be a finite-dimensional central division algebra over a field $F$, and let $w$ be a height one
valuation of $F$.
Assume that there exists a subring $\mathcal{B} \varsubsetneqq D$ such that
\begin{itemize}
\item[(i)]
$\mathcal{B}$ is open in $D$ with respect to the topology defined by the norm $\vv \ \vv_w$;
\item[(ii)]
there exists a positive integer $\ell$ such that $d \mathcal{B}d^{-1} \subseteq \mathcal{B}$
for all $d \in (D^{\times})^{\ell}:=\{x^{\ell}\mid x\in D^{\times}\}$.
\end{itemize}

\noindent
Then $w$ extends uniquely to a height one valuation $\widetilde{w}$ of
$D$ such that $\mathcal{B} \subseteq \mathcal{O}_{D , \widetilde{w}}$.
\end{prop}

\begin{thm}
\label{thm val2}
Let $D$ be a finite-dimensional separable division algebra over a field $k$ of finite transcendence degree over its prime field, and let $F=Z(D)$ be its center.
Let $N \subseteq D^{\times}$ be a normal subgroup containing $-1$ and of finite index.
Assume that $\varphi\colon N \to \Gamma$ a strongly leveled map such that
%
%
\begin{itemize}
\item[(i)]
$\gvp(N\cap k^{\times})$ is totally ordered.

\item[(ii)]
$N_{\geq0}$ contains a basis of $D$ over $k$.

\item[(iii)] for the subring $\mathcal{R}$ of $D$ generated by $N_{\ge 0}$, there exists
$\gamma \in \Gamma_{> 0}$ such that  $\mathcal{R} \cap N \subseteq N_{>-\gamma}$.
%
\end{itemize}
Then
\begin{enumerate}
\item
the restriction $\varphi_k =\varphi \restr_{N \cap k^{\times}}$ is a strong valuation-like map;

\item
there exists a height one valuation $v$ of $k$ associated with
$\varphi_k$ such that $N\cap k^{\times}$ is open in $k^\times$ in the $v$-adic topology;


\item
there exists a finite non-empty set $T$ of valuations on $F$ extending $v$  such that
$\vert T\vert\le [F\colon k],$ and such that
\begin{enumerate}
\item
each $w \in T$ uniquely extends to a valuation $\widetilde w$ of $D$ associated with $\varphi;$

\item
$N$ is open in $D^\times$ in the $\widetilde T$-adic topology,
where $\widetilde T = \{\widetilde{w} \: \vert \: w \in T \}$.
\end{enumerate}
\end{enumerate}
\end{thm}
\begin{proof}
Parts (1) and (2) hold by Theorem \ref{thm val}, using hypothesis (i).

Let $S$ be the set of all extensions of $v$ to $F$.
Then $S$ is non-empty, and as already noted, $|S|\leq[F:k]<\infty$.
Let $w_1,\ldots,w_r$ be the distinct valuations in $S$.
Let $k_v$ be the completion of $k$ with respect to $v$.
Since $F/k$ is separable, we have
\[
F \otimes_k k_v \cong \prod_{i = 1}^r F_{w_i}
\]
(see \cite{Bo}, Ch.\ 6, \S8, n{$^\circ$}2, Cor.\ 2), and therefore
\[
D \otimes_k k_v = D \otimes_F (F \otimes_k k_v)
\cong D \otimes_F \left(\prod_{i = 1}^r F_{w_i} \right) = \prod_{i = 1}^r D_{w_i} = D_S.
\]
This isomorphism commutes with the natural map
$D\to D \otimes_k k_v$, $d\mapsto d\otimes1$, and the diagonal map $\iota_S\colon D\to D_S$.
We extend the norm $\vv \ \vv_v$ on $D$ to $D \otimes_k k_v$.
Since any two norms on a finite-dimensional vector space over a complete normed field are equivalent
(see e.g., \cite{LangAlg}, p.\ 470, Prop.\ 2.2), the $\vv \ \vv_v$-topology on $D \otimes_k k_v$ coincides under this
isomorphism with the product topology $\tau_S$ on $D_S$.
Thus $\tau_S$ restricts to the topology $\tau_v$ on $D$.

Take a basis $\nu_1, \ldots , \nu_m$ of $D$ over $k$ which is contained in $N_{\geq0}$.
Then $D=k\nu_1+\ldots+k\nu_m$, with the $\tau_v$-topology, is homeomorphic
to the direct product of $m$ copies of $k$, with the $v$-adic topology, via the coordinate map.
Let $\mathcal{R}$ (resp., $\mathcal{R}_0$) denote the subring of $D$ (resp., $k$) generated by $N_{\ge0}$
(resp., $(N\cap k^\times)_{\ge0}$).
By (1) and (2),
Theorem \ref{thm val comm}(3) applies and we conclude
that the subring $\mathcal{R}_0$ of $k$  is $v$-adically open.
We have
\[
\mathcal{R}_0 \nu_1 + \cdots + \mathcal{R}_0 \nu_m \subseteq \mathcal{R},
\]
so $\mathcal{R}$ is open in $D$ with respect to $\tau_v$.

Further, let  $\widebar{\mathcal{R}}$ be the $\tau_S$-closure of $\mathcal{R}$ in $D_S$.
Then $\iota_S^{-1}(\widebar{\mathcal{R}})$ is the $\tau_v$-closure of $\mathcal{R}$ in $D$.
Being a $\tau_v$-open ring, $\mathcal{R}$ is also $\tau_v$-closed in $D$,
and therefore  $\mathcal{R}=\iota_S^{-1}(\widebar{\mathcal{R}})$.
Also, since $\mathcal{R}$ is $\tau_v$-open in $D$, the closure $\widebar{\mathcal{R}}$ is $\tau_S$-open in $D_S$
(indeed, for an open ball $B_D(0,\epsilon)$ in $\mathcal{R}$ one has
$B_{D\tensor_k k_v}(0,\epsilon)\subseteq\overline{B_D(0,\epsilon)}\subseteq \widebar{\mathcal{R}}$).

Next let $S_0$ denote the set of all $w \in S$ such that $\mathrm{pr}_w(\widebar{\mathcal{R}})$ is unbounded,
and set $T:=S\sminus S_0$.
By Lemma \ref{lem pr},
$\widebar {\mathcal{R}}=\mathrm{pr}_T(\widebar{\mathcal{R}}) \times D_{S_0}$.
Consequently
\begin{equation}
\label{inverse image of diag map}
\mathcal{R}=\iota_S^{-1}(\widebar{\mathcal{R}})
=\iota_S^{-1}(\mathrm{pr}_T(\widebar{\mathcal{R}})\times D_{S_0})
=\iota_T^{-1}(\mathrm{pr}_T(\widebar{\mathcal{R}})).
\end{equation}
In case $T = \emptyset$, this would mean $\mathcal{R} = D$, which contradicts [RSS], Prop. 4.3(1). Thus, $T$ is not empty. Furthermore, since $\mathrm{pr}_T(\overline{\mathcal{R}})$ is open in $D_T$, we obtain from (10.2) that $\mathcal{R}$ is open in $D$ with respect to $\tau_T$, as required.

Next we show that each $w \in T$ extends uniquely to $D$.
Indeed, since $\widebar{\mathcal{R}}$ is $\tau_S$-open in $D_S$,  the subring
\[
\mathcal{R}(w) := \mathrm{pr}_w(\widebar{\mathcal{R}}) \cap D
\]
is $\tau_v$-open in $D$.
Since $\mathrm{pr}_w(\overline{\mathcal{R}})$ is bounded, while $D$ is obviously unbounded, $\mathcal{R}(w)\neq D$.
Furthermore, being generated by $N_{\geq 0}$, the subring $\mathcal{R}$ is invariant under
conjugation by any element of $N$, so the same is true for
$\widebar{\mathcal{R}}$ and consequently for $\mathcal{R}(w)$.
Since $(D^{\times})^{\ell}\subseteq N$, for $\ell=[D^{\times}:N]$,
we can apply Proposition \ref{thm RSS} with $\mathcal{B}=\mathcal{R}(w)$
 and deduce that $w$ extends uniquely
to a height one valuation $\widetilde w$ of $D$ such that $\mathcal{R}(w)\subseteq \mathcal{O}_{D,\widetilde w}$.
In particular, $N_{\geq0}\subseteq \mathcal{O}_{D,\widetilde w}$, so
$\widetilde w$ is associated with $\varphi$.

The fact that $w$ extends to $D$ implies that $D_w = D \otimes_F F_w$ is a division algebra (see \cite{Cohn} or
\cite{Wad}, Th.\ 2.3).
Further,  $D_w$ can be identified with the completion of $D$ with respect to $\widetilde{w}$.
Since $D_w$ is finite-dimensional over the complete field $F_w$, the norm $\vv \ \vv_w$ and
the norm corresponding to $\widetilde{w}$ are equivalent, and hence induce the same topology
on $D_w$.
It follows that the $\tau_T$-topology on $D$ coincides with its $\widetilde{T}$-adic topology.
Therefore,  $\mathcal{R}$ is also $\widetilde{T}$-adically open.

Next, by (iii), there exists $\gamma \in
\Gamma_{> 0}$ such that $\mathcal{R} \cap N \subseteq N_{>-\gamma}$.
Hence, Proposition \ref{prop open} applies to $\mathcal{M} = \mathcal{R}$
and gives the openness of $N$ in the $\widetilde{T}$-adic topology.
\end{proof}

%

Assumption (iii) of Theorem \ref{thm val2} is satisfied in the following important situation:

\begin{lemma}[\cite{RSS}, Theorem 5.8(3)]\label{RSS Th 5.8(3)}
Let $x,y\in D^\times\sminus N$ satisfy $d(x^*,y^*)\geq3$ and let $\gC=\gC_{y^*}$ and $\ge=\ge_{y^*}$.
For the subring $\mathcal{R}$ of $D$ generated by $N_{\ge_0}$, there exists
$\gamma \in \Gamma_{> 0}$ such that  $\mathcal{R} \cap N \subseteq N_{>-\gamma}$.
\end{lemma}


We are now in a position to prove Theorems A and  C of the Introduction and Theorem \ref{thm v-3.5}.

%
%
%
\begin{proof}[Proof of Theorem C]
Take $x,y\in D^\times\setminus N$ with $d(x^*,y^*)\geq3$ and set $\ge=\ge_{y^*}$.
By Theorem \ref{thm existence of s-level in diam 3}, $\gvp=\gvp_{y^*}\colon N\to\gC$ is strongly leveled map.
Now Theorem \ref{thm val} gives (1) and (2).
Moreover, the valuation $v$ in (2) is associated with $\varphi\restr_{N\cap k^\times}$.
By Corollary \ref{cor N(a inverse) negative}(3), $N_{\geq0}$ contains a basis of $D$ over $k$.
By Lemma \ref{RSS Th 5.8(3)},  there exists $\gamma \in \Gamma_{> 0}$ such that  $\mathcal{R} \cap N \subset N_{>-\gamma}$,
where  $\mathcal{R}$ is the subring of $D$ generated by $N_{\ge_0}$.
Therefore (3) follows from Theorem \ref{thm val2}.
\end{proof}

\begin{proof}[Proof of Theorem \ref{thm v-3.5}]
Since ${\rm diam}(\gD)\ge 3,$ and by Theorem \ref{thm 3.5}, we can pick a nonidentity element $y^*\in\Dt/N$ such
that $\gvp=\gvp_{y^*}\colon N\to\gC$ is a strongly leveled map as in Theorem B(1), and
such that the subgroup $\gvp_{y^*}(N\cap k^{\times})$ of\, $\gC_{y^*}$ is totally ordered.
Hence Theorem \ref{thm v-3.5} follows from Theorem C.
\end{proof}

We conclude this section with the proof of Theorem A.

\noindent
\begin{proof}[Proof of Theorem A]
By Theorem \ref{thm existence of VL in diam 4}, we can pick a nonidentity element $y^*\in\Dt/N$ such
that $\gvp=\gvp_{y^*}\colon N\to\gC$ is  a strong valuation-like map.

Suppose first that $D=K$ is abelian.  Then, by Theorem \ref{thm val comm}(1), the assertion
of Theorem A holds.

Next assume that $D$ is a finite-dimensional division algebra over a field $K$ of finite transcendence degree over its prime field.
Of course $k$ is infinite since $D$ is.
We may assume that $k=F=Z(D)$.
Since $\gvp$ is  a strong valuation-like map, the subgroup \mbox{$\gvp(N\cap F^{\times})\subseteq \gC$}
is totally ordered.
Hence the assertion of Theorem A holds by Theorem C(3)
(with $k=F$, and note that $\vert T\vert=\vert\widetilde{T}\vert=1$, since $k=F$).
%
%
%
\end{proof}


\section{Examples}\label{sect examples}

\subsection*{11A. \ Constructions of V-graphs}

\begin{example}\label{eg non-commutative milnor}
One can extend the notion of the  relative Milnor $K$-rings from the case of commutative fields, as in \cite[Part IV]{Efr_book}, to our general non-commutative context as follows.
Let $D$ be a division ring and let $N$ be a normal subgroup of $D^\times$.
We stress however that, when $N=1$, this definition is not the more common one, as in e.g.
\cite{Rehmann}, \cite{RehmannStuhler}.
For $r\geq1$ we define the degree $r$ Milnor $K$-group $K^M_r(D)/N$ of $D$ relative to $N$ as the abelian group generated by all $r$-tuples $\langle a_1^*,\ldots,a_r^*\rangle$ in $(D^\times/N)^r$, subject to the following
defining relations:
\begin{enumerate}
\item[(a)]
\textsl{multi-linearity:}
\[
\langle a_1^*,\ldots, a^*_i(a'_i)^*,\ldots,a_r^*\rangle=\langle a_1^*,\ldots, a^*_i,\ldots,a_r^*\rangle+\langle a_1^*,\ldots,(a'_i)^*,\ldots,a_r^*\rangle
\]
\item[(b)]
\textsl{Steinberg relations:} $\langle a_1^*,\ldots,a_r^*\rangle=0$ whenever $1\in a_i^*+a_j^*$ for some distinct $i,j$.
\end{enumerate}
In particular,  $K^M_1(D)/N$ is the maximal abelian quotient of $D^\times/N$.
We also define $K^M_0(D)/N=\zz$.

We equip $K^M_*(D)/N=\bigoplus_{r=0}^\infty K^M_r(D)/N$ with a graded ring structure, where the product is induced by concatenation.
Denote the natural multi-linear map $(D^\times/N)^r\to K^M_r(D^\times)/N$ by $(a^*_1,\ldots,a^*_r)\mapsto \{a_1,\ldots, a_r\}_N$.
Note that it is well-defined.

Now let $a\in D^\times$.
We show that $\{a,-a\}_N=0$.
This is trivial for $a=1$.
For $a\neq1$ the identity
\[
-a=(1-a)(a^{-1}(a-1))^{-1}=(1-a)(1-a^{-1})^{-1}
\]
implies that
\[
\begin{split}
\{a,-a\}_N&=\{a,1-a\}_N+\{a,(1-a^{-1})^{-1}\}_N \\
&=\{a,1-a\}_N+\{a^{-1},1-a^{-1}\}_N=0,
\end{split}
\]
as claimed.
In the terminology of \cite[\S23.1]{Efr_book} this shows that $K^M_*(D)/N$ is a $\kappa$-structure with the distinguished element $(-1)^*$.
It is now standard to see that $K^M_*(D)/N$ is anti-commutative (see \cite[Lemma 23.1.2]{Efr_book}); indeed, for $a,b\in D^\times$
we have
\[
\begin{split}
\{a,b\}_N+\{b,a\}_N&=\{ab,ab\}_N-\{a,a\}_N-\{b,b\}_N \\
&=\{ab,-1\}_N-\{a,-1\}_N-\{b,-1\}_N=0.
\end{split}
\]

Now assume that $(D^\times:N)<\infty$ and $-1\in N$.
Extending the construction in \cite{Efr4}, we define the \textsl{Milnor $K$-graph of $D$ relative to $N$} as the graph whose vertices are the non-trivial cosets in $D^\times/N$, and where vertices $a^*$ and $b^*$ are connected by an edge if and only if $\{a,b\}_N=0$ in $K^M_2(D)/N$ (this relation is symmetric by the anti-commutativity).
By construction, this graph satisfies condition (V1$'$).
Exactly as in \cite[Lemma 2.1]{Efr4} one shows using the bilinearity of $\{\cdot,\cdot\}_N$ that it satisfies (V2) and (V3).
\end{example}

\begin{example}\label{eg min graph}
Let $H$ be a group and let $\gTH$ be a graph on $H\sminus\{1\}$
with distance function $d$.
For $h\in H\sminus\{1\}$ denote
\[
\gTH_h:=\{g\in H\sminus\{1\}\mid d(h,g)\le 1\}\cup\{1\}.
\]
We say that $\gTH$ is a {\it centralizer-graph} if $\gTH_h$ is
a subgroup of $H$ contained in the centralizer $C_H(h)$, for all $h\in H$.

Notice that if we are given  subgroups $C_h\le C_H(h)$, $h\in H\sminus\{1\}$,
we can construct the minimal centralizer-graph $\gTH$
on $H\sminus\{1\}$ such that $C_h\le\gTH_h$, for all $h\in H\sminus\{1\}$.
This is just the intersection of all centralizer-graphs $\gTH'$
such that $C_h\le\gTH'_h$, $h\in H\sminus\{1\}$.

For each non-identity $a^*\in \Dt/N$ (notation as in subsection \ref{sub vg} of the Introduction),
let $C_{a^*}:= \lan (a+n)^*, (a^{-1}+n)^*\mid n\in N\ran$.
Let $\gD$ be the minimal centralizer-graph such that
$C_{a^*}\le \gD_{a^*}$, for all $a^*\in \Dt/N$.
Then $\gD$ is a  V-graph, and it is the minimal centralizer-graph
which is a  V-graph.
\end{example}
\medskip

\subsection*{11B. \ The bound $\mathrm{diam}\geq 3$ is sharp in Theorem C(3)}\hfill
\medskip

\noindent
In Theorem B(2), the bound $3$ on the diameter is sharp.
An example showing this for a commuting graph of the quaternion algebra $(-1,-1/\qq)$ was given in \cite[\S5]{R}.
The following analogous example shows this for a Milnor $K$-graph over $\qq$.

\begin{example}
Fix an odd integer $m$ and a prime number $l$ such that $\gcd(l-1,m)>1$.
Thus $\ff_l^\times\neq(\ff_l^\times)^m$.
Similarly to \cite[\S5]{R} and \cite[p.\ 582]{RS} we consider the homomorphism
\[
h\colon \qq^\times\to\zz, \quad
\pm{\textstyle\prod_i}p_i^{d_i}\mapsto{\textstyle\sum_i}d_i,
\]
where the $p_i$ denote the distinct primes, $d_i\in\zz$, and $d_i=0$ almost always.
Let $H=h^{-1}(m\zz)$ and let  $U=(\qq^\times)^m(1+\frakm_{v_l})$, where $v_l$ is the $l$-adic valuation on $\qq$.
One has a split exact sequence
\[
1\to\ff_l^\times/(\ff_l^\times)^m\to \qq^\times/U\xrightarrow{v_l}\zz/m\to0,
\]
(see e.g.\ \cite[(3.2.7)]{Efr_book}).
Hence $(\qq^\times:U)<\infty$.
Setting $N=H\cap U$, we obtain that
\[
-1\in(\qq^\times)^m\leq N\leq \qq^\times,
\quad (\qq^\times:N)<\infty.
\]

We show that $N$ cannot be open with respect to any finite set of non-trivial valuations on $\qq$.
To this end let $q_1,\ldots, q_t$ be distinct prime numbers, let $r_1,\ldots, r_t$ be positive integers, and consider the $q_i$-adic valuations $v_{q_i}$, $i=1,\ldots, t$.
We need to show that $N$ does not contain the congruence subgroup
\[
W=1+q_1^{r_1}\mathcal{O}_{v_{q_1}}\cap\cdots\cap q_t^{r_t}\mathcal{O}_{v_{q_t}}.
\]
Indeed, Dirichlet's theorem on primes in arithmetic progressions yields a positive integer $k$ with $p=1+kq_1^{r_1}\cdots q_t^{r_t}$ prime.
Thus $h(p)=1$, so $p\notin H$.
Also note that $p\in W$.
Consequently, $W\not\subseteq H$, whence $W\not\subseteq N$, as desired.

Finally we compute the diameter of the graph of $K^M_*(\qq)/N$.
In view of Theorem C (with $D=F=k=\qq$) and what we have just seen, the diameter is at most $2$.
To show that it is exactly $2$ we need to verify that $K^M_2(\qq)/N\neq0$.
Now there is a canonical isomorphism of $\kappa$-structures
\[
K^M_*(\qq)/U\cong (K^M_*(\ff_l)/(\ff_l^\times)^m)[\zz/m]
\]
\cite[Th.\ 26.1.2]{Efr_book}.
It follows that  $K^M_2(\qq)/U\neq0$  (see \cite[Example 23.2.5]{Efr_book}).
Applying the canonical epimorphism $\Res\colon K^M_*(\qq)/N\to K^M_*(\qq)/U$ we get that $K^M_2(\qq)/N\neq0$ as well.
More concretely, take an $l$-adic unit $a$ in $\zz$ whose residue $\bar a\in\ff_l^\times$ is not an $m$th power.
Then $aU\tensor lU\notin\St_{\qq,2}(U)$, and therefore also $aN\tensor lN\notin\St_{\qq,2}(N)$, i.e., $\{a,l\}_N\neq0$ in $K^M_2(\qq)/N$.

We remark that, by contrast, the diameter of $K^M_*(\qq)/H$ is $1$.
Indeed, $\qq^\times/H$ is cyclic and $H$ is not contained in any ordering (since $-1\in H$).
Hence, by \cite[Th.\ 25.2.1]{Efr_book}, $K^M_2(\qq)/H=0$, and the assertion follows.
\end{example}
\medskip

\subsection*{11C. \ Examples for Theorem C and property ($3\half$)}\hfill
\medskip

\noindent
Examples \ref{semi-local example for commuting graph} and \ref{semi-local example for K-graphs}
below show that Theorem C covers a more general situation than
the previous works \cite{R}, \cite{RS}, \cite{RSS}, \cite{Efr4}.
These examples also illustrate property ($3\half$), and shows that Theorem \ref{thm 3.5}
covers a more general situation than \cite[Theorem 5.8]{RSS}.

\begin{example}
\label{semi-local example for commuting graph}
First we study the following local situation.
Let $D$ be a  division algebra over a field $F$ and $\widetilde w$ a discrete valuation on $D$
with uniformizer $\pi$.
Let $U=O_{D,\widetilde w}^\times$ be the group of $\widetilde w$-units in $D$,
let $U^{(1)} = U^{(1)}_{\widetilde w} = 1 + \mathfrak{m}_{D , \widetilde w}$ be its first congruence subgroup,
and let $\overline{D}=\overline D_{\widetilde w}$ be the residue field.
Suppose that $e$ is a positive integer with $\pi^e\in Z(D)$.
Then
\[
N = \langle \pi^e ,U^{(1)}\rangle = \langle \pi^e \rangle \times U^{(1)}
\]
is a normal subgroup of $D^\times$.
One has an isomorphism
\[
U/(U\cap N)=U/U^{(1)} \cong \overline{D}^\times.
\]
Therefore $\widetilde w$ induces an exact sequence
\begin{equation}
\label{exact sequence for tilde w}
1 \to  \overline{D}^\times \to D^\times/N \to \mathbb{Z}/e\mathbb{Z} \to 0.
\end{equation}
Assume further that the residue characteristic of $\widetilde w$ is $2$.
Then $-1\in U^{(1)}\subseteq N$.

Let $N(y)$ be as in \S\ref{sect the ordered gp}.

\begin{lemma}
\label{L:N(b)}
Let $y \in D^\times\sminus N$.
Then
$N(y) = \{ n\in N \ \vert \ \widetilde w(n) < \widetilde w(y) \}$.
\end{lemma}
\begin{proof}
Set $s=\widetilde w(y)$.
Consider an arbitrary $n\in N$ and write it in the form $n = \pi^tu$,
with $t=\widetilde w(n) \in e\mathbb{Z}$ and $u \in U^{(1)}$.

\vskip1mm

{\it Case 1.} $t<s$.
Then $\widetilde w(\pi^{-t}y)>0$, so $\pi^{-t}y+u\in U^{(1)}$, implying that
$y+n=\pi^t(\pi^{-t}y+u)\in N$ and $n\in N(y)$.

\vskip1mm

{\it Case 2.} $t=s$.
Then $-\pi^{-t}y\in U\sminus U^{(1)}$, so $\pi^{-t}y+u\in U\sminus U^{(1)}$,
and therefore $y+n=\pi^t(\pi^{-t}y+u)\not\in N$, i.e., $n\not\in N(y)$.

\vskip1mm

{\it Case 3.} $t >s$.
Then $\widetilde w(\pi^{t-s}u)>0$.
Since $\pi^{-s}y\in U\sminus U^{(1)}$, we have $\pi^{-s}y+\pi^{t-s}u\in U\sminus U^{(1)}$,
and consequently $y+n\not\in N$ and $n\not\in N(y)$.
\end{proof}

\medskip

Next we consider the following semi-local situation.
Let $\widetilde w_1,\ldots, \widetilde w_r$ be distinct discrete
valuations on the division algebra $D$ with residue characteristic $2$.
Suppose that $\pi\in D$ is a common uniformizer for $\widetilde w_1,\ldots,\widetilde w_r$, and that $\pi^e\in Z(D)$ with $e \ge1$.
Set $N_i=\langle\pi^e\rangle\times U^{(1)}_{\widetilde w_i}$ and $N=N_1\cap\cdots\cap N_r$.
By the weak approximation theorem,
$D^\times/N\cong \prod_{i=1}^r D^\times/N_i$ via the diagonal map.
In the terminology of \S\ref{sect the ordered gp}, we set $U=U_{\pi^*}$ and $\Gamma=N/U$.

\begin{lemma}
\label{Zr}
There is an isomorphism of partially ordered groups
\[
 \Gamma=N/U \  \smash{\mathop{\longrightarrow}\limits^{\sim}} \  \mathbb{Z}^r, \quad
nU \mapsto (\tfrac1e\widetilde w _1(n),\ldots, \tfrac1e\widetilde w_r(n)),
\]
where $\mathbb{Z}^r$ is endowed with the product partial order relation.
Further,
\[
U=U^{(1)}_{\widetilde w_1}\cap \cdots\cap U^{(1)}_{\widetilde w_r}.
\]
\end{lemma}
\begin{proof}
By Lemma \ref{L:N(b)} with $y=\pi$,
\begin{equation}
\label{N(y) in a semi-local example}
N(\pi) = N_1(\pi)\cap\cdots\cap N_r(\pi)
 = \{ t \in N \ \vert \ \widetilde w_1(t),\ldots,\widetilde w_r(t)\leq0  \}.
\end{equation}
Hence $1\in N(\pi)$, and every $n\in U^{(1)}_{\widetilde w_1}\cap \cdots\cap U^{(1)}_{\widetilde w_r}$ satisfies $nN(\pi)=N(\pi)$.
Conversely,  if $n\in N$ satisfies $nN(\pi)=N(\pi)$, then $n,n^{-1}\in N(\pi)$, so  $\widetilde w_i(n) = 0$ for every  $i$.
It follows that
$U = U^{(1)}_{\widetilde w_1}\cap\cdots\cap U^{(1)}_{\widetilde w_r}$.

By the weak approximation,
\begin{equation}
\label{isom by weak approx}
\Gamma = N/U \cong N_1/U^{(1)}_{\widetilde w_1} \times\cdots\times N_r/U^{(1)}_{\widetilde w_r}\cong\mathbb{Z}^r,
\end{equation}
where the right isomorphism is induced by $(\tfrac1e\widetilde w_1,\ldots,\tfrac1e\widetilde w_r)$.

Moreover, if $m,n\in N$ and $N(m\pi)\subseteq N(n\pi)$, then $mN(\pi)\subseteq nN(\pi)$,
and (\ref{N(y) in a semi-local example}) implies that
$\widetilde w_i(m)\leq\widetilde w_i(n)$ for $i=1,2,\ldots,r$.
Therefore (\ref{isom by weak approx}) is an isomorphism of partially ordered groups.
\end{proof}

As a concrete example, take $d\in \mathbb{Q}$ such that $\sqrt d\in \mathbb{Q}_2\sminus\mathbb{Q}$;
e.g.,  we may take $d=17$ (see \cite{Efr_book}, Prop.\ 18.2.1).
Let $k$ be $\mathbb{Q}$, or more generally, a subfield of $\mathbb{Q}_2$ which does not contain $\sqrt d$.
Let $F=k(\sqrt d)$, and let $\sigma$ be the nontrivial automorphism in $\mathrm{Gal}(F/k)$.
Let $v,w_1$ be the restrictions of the $2$-adic valuation of $\mathbb{Q}_2$ to $k,F$, respectively.
Also let $w_2=w_1\circ\sigma$.
Then $w_1,w_2$ are unramified over $v$ and have residue field $\mathbb{F}_2$.
Further, they are distinct, and are the only extensions to $F$.
Consider the quaternion algebra  $D=(-1,-1/F)$ over $F$ with its standard basis ${\bf1},{\bf i},{\bf j},{\bf k}$, and take $e=2$.
Then $w_1,w_2$ extend to valuations $\widetilde w_1,\widetilde w_2$, respectively, on $D$ with
$(\widetilde w_i(D^\times):w_i(F^\times))=2$ and $\overline{D}_{\widetilde w_i}=\mathbb{F}_4$,  $i=1,2$.
Specifically, let
\[
\pi={\bf i}+\mathbf{j}, \quad a=\frac12(-{\bf1}+{\bf i}+{\bf j}+{\bf k})=-\frac12({\bf1}-{\bf j})({\bf1}-{\bf i}).
\]
Note that $\pi^2=-2$ and $a^2+a+1=0$.
Then $\pi$ is a uniformizer for $\widetilde w_i$  and $a$ is a $\widetilde w_i$-unit with residue not in $\mathbb{F}_2$, $i=1,2$.
Let
\[
N_i=\langle\pi^2\rangle\times U^{(1)}_{\widetilde w_i}=\langle-2\rangle\times U^{(1)}_{\widetilde w_i}, \quad
N=N_1\cap N_2.
\]

By the exact sequence (\ref{exact sequence for tilde w}), the group $D^\times/N_i$ has order $6$ (where $i=1,2$).
We show that it is noncommutative.
Indeed,
\[
\begin{split}
\pi a\pi^{-1} a^{-1}&=\frac12({\bf1}-{\bf i}+{\bf j}+{\bf k})=\frac12({\bf1}+{\bf j})({\bf1}-{\bf i}) , \quad  \\
\pi a\pi^{-1} a^{-1}-1&=\frac12(-{\bf1}-{\bf i}+{\bf j}+{\bf k})=-\frac12({\bf1}+{\bf i})({\bf1}-{\bf j})
\end{split}
\]
are in $U_{\widetilde w_i}$.
Thus $\pi a\pi^{-1} a^{-1}\in U_{\widetilde w_i}\sminus U^{(1)}_{\widetilde w_i}$,
so $aN_i$ and $\pi N_i$ do not commute in $D^\times/N_i$
(this also follows from the general fact that, if
$\cald$ is a finite-dimensional central division algebra over a discretely valued field $\calk$ with a perfect residue field $\overline{\calk}$,
then the center $Z(\overline{\cald})$ is a cyclic Galois extension of $\overline{\calk}$, and a uniformizer $\Pi$ in $\cald$ induces a generator of $\hbox{Gal}(Z(\overline\cald)/\overline\calk)$, cf.~\cite[Prop.~2.5]{Wad}.
In our situation,  $\bar a\in Z(\overline\cald)\sminus\overline\calk=\mathbb{F}_4\sminus\mathbb{F}_2$, so this generator must act non-trivially on $\bar a$).

Consequently, $D^\times/N_1\cong D^\times/N_2\cong S_3$ and
\begin{equation}
\label{S3  x S3}
D^\times/N\cong S_3\times S_3.
\end{equation}
Observe that the diameter of the commuting graph of $S_3 \times S_3$ is $3$.
Indeed, any $(\sigma_1 , \sigma_2) , (\tau_1 , \tau_2) \in S_3 \times S_3$ are connected in the
commuting graph by the following path of length $3$:
\[
(\sigma_1,\sigma_2) \ , \ (1 , \sigma_2) \ , \ (\tau_1,1) \ , \ (\tau_1, \tau_2)
\]
provided that $\sigma_2 ,\tau_1 \neq 1$; other cases are considered similarly.
On the other hand, if $\sigma_1,\sigma_2 \in S_3$ are transpositions
and $\tau_1,\tau_2 \in S_3$ are 3-cycles, then $(\sigma_1,\sigma_2)$
and $(\tau_1,\tau_2)$ are not at distance $\le 2$.

Now $\pi N_i$ has order $2$ and $aN_i$ has order $3$  in $D^\times/N_i\cong S_3$.
Hence $\pi N,aN\in D^\times/N$ correspond under (\ref{S3 x S3}) to elements of the form
$(\sigma_1,\sigma_2)$, $(\tau_1 , \tau_2)$, respectively, where $\sigma_1,\sigma_2$
are transpositions and $\tau_1,\tau_2$ are 3-cycles.
It follows that $d(aN , \pi N)=3$.

Therefore the assumptions of Theorem C are satisfied
with $k$ a field of finite transcendence degree as above, $x=a$, $y=\pi$,
and with $\varphi=\varphi_{y^*}\colon N\to \Gamma=N/(U_{\widetilde w_1}\cap U_{\widetilde w_2})\cong\mathbb{Z}^2$
(see Lemma \ref{Zr}).
Then $\varphi(nU)=(\tfrac12 v(n),\tfrac12 v(n))$ for $n\in N\cap k^\times$,
so $\varphi(N\cap k^\times)\cong\mathbb{Z}$ is indeed totally ordered.
This also shows that $\varphi\restr_{N\cap k^\times}$ is a strong valuation-like map,
and $N\cap k^\times$ is open in the $v$-adic topology on $k$, in accordance with (1) and (2) of Theorem C.
By construction, (3) of Theorem C holds with $T=\{w_1,w_2\}$ and $\widetilde T=\{\widetilde w_1,\widetilde w_2\}$.

On the other hand, the results of \cite{RS}, \cite{RSS} do not apply to characterize $N$.
Namely, by Lemma \ref{lemma on coarsenings}, $N$ is not $\widetilde u$-adically open for any nontrivial valuation $\widetilde u$ on $D$.
Technically, the results of \cite{R}  may not apply either as $k$ and $F$ do not have to be number fields
(although some methods of \cite{R} were instrumental in proving Theorem C).
\end{example}

This construction also provides an example of a situation where property (3$\half$) of \S7 applies.
More precisely, the nontrivial automorphism $\sigma$ of ${\rm Gal}(F/k)$ extends
to an automorphism $\sigma_D$ of order $2$ of the algebra $D$.
This automorphism switches the valuations $w_1$ and $w_2$ of $F$,
the valuations $\widetilde w_1$ and $\widetilde w_2$ of $D$, and the subgroups $N_1$ and $N_2$ of $D^\times$.
The subgroup $N$ of $D^\times$ is invariant under $\sigma_D$, and $\sigma_D$ switches
the factors in $D^{\times}/N = S_3 \times S_3$.
Consider elements $x^* = (\sigma_1 , \sigma_2)$  and $y^* = (\tau_1 , \tau_2)$ as above.
We claim that $D^\times/N$ satisfies (3$\half$) with respect to $x^*,y^*$,
 the group $\Sigma=\{{\rm id},\sigma_D\}$, and $M=N\cap k^\times$.
Indeed, essentially the only path of length 3 between $x^*$ and $y^*$ is
\[
x^*, \  (\sigma_1 , 1), \  (1 , \tau_2), \  y^*.
\]
But  $d(\sigma_D(x^*) , y^*) = 3$, and $\sigma_D(\sigma_1,1) = (1, \sigma_1)$ and $(1,\tau_2)$ do not commute.

\medskip

\begin{example}
\label{semi-local example for K-graphs}
We give a similar construction for a Milnor $K$-graph.
Let $p$ be a number such that $p\equiv1\pmod4$.
Let $F/k$ be a nontrivial Galois field extension such that $F\subseteq\mathbb{Q}_p$.
Let $v,w$ be the restrictions of the $p$-adic valuation on $\mathbb{Q}_p$ to $k,F$, respectively.
The extension $(F,w)/(k,v)$ is immediate, i.e., has the same value group $\mathbb{Z}$ and residue field $\mathbb{F}_p$.
Since $v$ is discrete, this implies that there are exactly $[F:k]$ extensions of $v$ to $F$,
namely $w\circ\sigma$ with $\sigma\in\mathrm{Gal}(F/k)$ (see \cite{Efr_book}, Cor.\ 17.4.4).
Set $w_1=w$ and $w_2=w\circ\sigma$ with $\sigma\neq\mathrm{id}$.

Now set
\[
N_1=(F^\times)^2(1+\mathfrak{m}_{w_1}), \quad
N_2=(F^\times)^2(1+\mathfrak{m}_{w_2}), \quad
N=N_1\cap N_2.
\]
Since $p\equiv1\pmod4$ we have $-1\in(\mathbb{F}_p^\times)^2=(\overline F_{w_i}^\times)^2$, $i=1,2$, and therefore $-1\in N$.

We compute  the relative Milnor $K$-ring $K^M_*(F)/N$.
We use the terminology and results of \cite{Efr_book}, Part IV.
First, the graded ring $K^M_*(\mathbb{F}_p) /(\mathbb{F}_p^\times)^2$ is
$\mathbb{F}_p^\times/(\mathbb{F}_p^\times)^2\cong \mathbb{Z}/2\mathbb{Z}$  in degree $1$,
and is trivial in degrees $\geq2$ \cite{Efr_book}, Cor.\ 25.2.4.
Thus $K^M_*(\mathbb{F}_p) /(\mathbb{F}_p^\times)^2$ is the extension
${\bf0}[\mathbb{Z}/2\mathbb{Z}]$ of the trivial $\kappa$-structure $\bf0$ by the group $\mathbb{Z}/2\mathbb{Z}$
\cite{Efr_book}, Example 23.2.4.
Since $w(N_1)=2\mathbb{Z}$, \cite{Efr_book}, Th.\ 26.1.2 and Example 26.1.1(1), says that $K^M_*(F)/N_1$
is the extension $(K^M_*(\mathbb{F}_2)/(\mathbb{F}_p^\times)^2)[\mathbb{Z}/2\mathbb{Z}]$.
It follows that
\[
K^M_*(F)/N_1\cong({\bf0}[\mathbb{Z}/2\mathbb{Z}])[\mathbb{Z}/2\mathbb{Z}]
\cong{\bf0}[(\mathbb{Z}/2\mathbb{Z})^2]
\]
(see \cite{Efr_book}, Lemma 23.2.3).
Similarly, $K^M_*(F)/N_2\cong{\bf0}[(\mathbb{Z}/2\mathbb{Z})^2]$.

Since $w,w'$ are distinct and discrete, they are independent.
Therefore  \cite{Efr_book}, Cor.\ 28.2.4, shows that
\[
K^M_*(F)/N\cong(K^M_*(F)/N_1)\times(K^M_*(F)/N_2)\cong
({\bf0}[(\mathbb{Z}/2\mathbb{Z})^2])\times({\bf0}[(\mathbb{Z}/2\mathbb{Z})^2]).
\]
The Milnor $K$-graph of $F$ relative to $N$ was computed in \cite{Efr4}, \S7, and it has diameter $3$.

We may further take $k$ to be finitely generated over $\mathbb{Q}$.
By construction, $N$ is open in the $T$-adic topology, where $T=\{w_1,w_2\}$.
This is in accordance with Theorem 10.4, when we take $D=F$ and $T=\widetilde T$.

Finally, by Lemma \ref{lemma on coarsenings}, $N$ is not $u$-adically open with respect
to any single non-trivial valuation $u$ on $F$.

In this example as well property ($3\frac12$) holds, when we take $F/k$ to be an extension of degree $2$, $\Sigma=\mathrm{Gal}(F/k)$, and $M=N\cap k^\times$.
The generator $\sigma$ of $\mathrm{Gal}(F/k)$ switches $N_1,N_2$, and therefore fixes $N$.
For $i=1,2$, the group $F^\times/N_i$ is the degree 1 component of $K^M_*(F)/N_i\cong {\bf0}[(\mathbb{Z}/2\mathbb{Z})^2]$, and is isomorphic to $(\mathbb{Z}/2\mathbb{Z})^2$.
Let $\alpha,\beta,\gamma,\delta$ be generators of $F^\times/N$ such that $\alpha,\beta$ project to generators  $F^\times/N_1$ and $\gamma,\delta$ project to generators of  $F^\times/N_2$.
We may assume that $\sigma$ switches $\alpha,\gamma$ and $\beta,\delta$.
As was shown in \cite{Efr4}, \S7 (in a multiplicative notation), the vertices $\alpha+\gamma,\beta+\delta$
of the Milnor $K$-graph of $F$ relative to $N$ have distance $3$, and there are exactly two paths of length $3$ connecting them, namely
\[
\alpha+\gamma, \alpha,\delta,\beta+\delta \quad \mathrm{and} \quad
\alpha+\gamma, \gamma,\beta,\beta+\delta.
\]
Further, $\sigma(\alpha+\gamma)=\alpha+\gamma$, but the vertices $\gamma=\sigma(\alpha)$ and $\delta$
are not connected by an edge, and neither  are the vertices
$\alpha=\sigma(\gamma)$ and $\beta$.
Therefore property ($3\frac12$) holds for $x^*=\alpha+\gamma$ and $y^*=\beta+\delta$.
\end{example}

\subsection*{Acknowledgment.}
We warmly thank the referee for the thorough, quick, and extremely useful report, which helped us to correct
certain flaws, as well as to improve the exposition of the paper.


\end{document}